\documentclass[11pt,a4paper]{amsart}
\usepackage[left=20mm,right=20mm]{geometry}
\usepackage[english]{babel}
\usepackage[T1]{fontenc}
\usepackage[utf8]{inputenc}
\usepackage{amsthm}
\usepackage{amssymb}
\title[Exponential convergence in supercritical kinetically constrained models]{Exponential 
convergence to equilibrium in supercritical kinetically constrained models 
at high temperature}
\author{Laure \textsc{Marêché}}
\address{Laure Mar\^ech\'e, LPSM UMR 8001, Universit\'e Paris Diderot, Sorbonne Paris Cit\'e, CNRS, 
75013 Paris, France}
\email{mareche@lpsm.paris}
\thanks{I acknowledge the support of the ERC Starting Grant 680275 MALIG}
\theoremstyle{plain}
\newtheorem{theorem}{Theorem}
\newtheorem{proposition}[theorem]{Proposition}
\newtheorem{lemma}[theorem]{Lemma}
\theoremstyle{remark}
\newtheorem{remark}[theorem]{Remark}
\theoremstyle{definition}
\newtheorem{definition}[theorem]{Definition}
\usepackage{dsfont}
\usepackage{tikz}
\usetikzlibrary{decorations}
\usetikzlibrary{decorations.pathreplacing}
\usepackage{hyperref}

\begin{document}

\maketitle

\begin{center}
\begin{minipage}{0.8\textwidth}
\begin{small}
\textbf{Abstract:} Kinetically constrained models (KCMs) were introduced by physicists to 
  model the liquid-glass transition. They are interacting particle systems on $\mathds{Z}^d$ in which 
  each element of $\mathds{Z}^d$ can be in state 0 or 1 and tries to update its state to 0 at rate $q$ and to 1 at rate $1-q$, 
  provided that a constraint is satisfied. 
  In this article, we prove the first non-perturbative result of convergence to 
  equilibrium for KCMs with general constraints: for any KCM in the class 
  termed ``supercritical'' in dimension 1 and 2, when the initial configuration has product $\mathrm{Bernoulli}(1-q')$ 
  law with $q' \neq q$, the dynamics converges to equilibrium with exponential speed when $q$ is close enough to 1, 
  which corresponds to the high temperature regime. 

\medskip


\end{small}
\end{minipage}
\end{center}

\section{Introduction}
Kinetically constrained models (KCMs) are interacting particle systems on $\mathds{Z}^d$, in 
which each element (or \emph{site}) of $\mathds{Z}^d$ can be in state 0 or 1. Each site tries 
to update its state to 0 at rate $q$ and to 1 at rate $1-q$, with $q \in [0,1]$ fixed, but an update is accepted if and only if a 
\emph{constraint} is satisfied. This constraint is defined via an \emph{update family} $\mathcal{U}=\{X_1,\dots,X_m\}$, 
where $m \in \mathds{N}^*$ and the $X_i$, called \emph{update rules}, are finite nonempty subsets of 
$\mathds{Z}^d \setminus \{0\}$: the constraint is satisfied at a site $x$ if and only if there exists $X \in \mathcal{U}$ 
such that all the sites in $x+X$ have state zero. Since the constraint at a site does not depend on the state of the 
site, it can be easily checked that the product $\mathrm{Bernoulli}(1-q)$ measure, $\nu_q$, 
satisfies the detailed balance with respect to the dynamics, hence is reversible 
and invariant. $\nu_q$ is the \emph{equilibrium measure} of the dynamics. 

KCMs were introduced in the physics literature by Fredrickson and Andersen \cite{Fredrickson_et_al1984} 
to model the liquid-glass transition, an important open problem 
in condensed matter physics (see \cite{Ritort_et_al,Garrahan_et_al}). In addition to this physical interest, 
KCMs are also mathematically challenging, because the presence of the constraints make them very different from 
classical Glauber dynamics and prevents the use of most of the usual tools. 

One of the most important features of KCMs is the existence of blocked configurations. 
These blocked configurations imply that the equilibrium measure $\nu_q$ is not 
the only invariant measure, which complicates a lot the study of the out-of equilibrium behavior 
of KCMs; even the basic question of their convergence to $\nu_q$ remains open in most cases. 

Because of the blocked configurations, one cannot expect such a convergence to equilibrium for all initial 
laws. Initial measures particularly relevant for physicists are the $\nu_{q'}$ with $q' \neq q$ (see \cite{Leonard_et_al2007}). 
Indeed, $q$ is a measure of the temperature of the system: the closer $q$ is to 0, the lower the temperature is. 
Therefore, starting the dynamics with a configuration of law $\nu_{q'}$ means starting with a temperature different 
from the equilibrium temperature. In this case, KCMs are expected to converge to equilibrium 
with exponential speed as soon as no site is blocked for the dynamics in a configuration of law $\nu_{q}$ or $\nu_{q'}$. 
However, there have been few results in this direction so far 
(see \cite{Cancrini_et_al2010,Blondel_et_al2013,stretched_exponential_East-like,Mountford_FA1f,Mareche2019Est}),
and they have been restricted to particular update families or initial laws.

Furthermore, general update families have attracted a lot of attention in recent years. 
Indeed, there recently was a breakthrough in the study of a monotone 
deterministic counterpart of KCMs called bootstrap percolation. Bootstrap percolation is a discrete-time dynamics 
in which each site of $\mathds{Z}^d$ can be \emph{infected} or not; 
infected sites are the bootstrap percolation equivalent of sites at zero. 
To define it, we fix an update family $\mathcal{U}$ and choose a set $A_0$ of initially 
infected sites; then for any $t \in \mathds{N}^*$, the set of sites that are infected at time $t$ is
\[
A_t = A_{t-1} \cup \{x \in \mathds{Z}^d \,|\, \exists X \in \mathcal{U}, x+X \subset A_{t-1}\},
\]
which means that the sites that were infected at time $t-1$ remain infected at time $t$ and a
site $x$ that was not infected at time $t-1$ becomes infected at time $t$ if and only if there 
exists $X \in \mathcal{U}$ such that all sites of $x + X$ are infected at time $t-1$. 
Until recently, bootstrap percolation had only been considered with particular update families, 
but the study of general update families was opened by Bollobás, Smith and Uzzell in \cite{Bollobas_et_al2015}. 
Along with Balister, Bollobás, Przykucki and Smith \cite{Balister_et_al2016}, they proved that
general update families satisfy the following universality result: in dimension 2, they 
can be sorted into three classes, \emph{supercritical}, \emph{critical} and \emph{subcritical} 
(see definition \ref{def_universality_classes}), which display different 
behaviors (their result for the critical class was later refined by Bollobás, Duminil-Copin, Morris and Smith 
in \cite{Bollobas_et_al2017}). 

These works opened the study of KCMs with general update families. 
In \cite{MMT,lbounds_infection_time,Hartarsky_et_al2019,Hartarsky_et_al2019bis}, Hartarsky, Martinelli, Morris, 
Toninelli and the author showed that the grouping of two-dimensional update families 
into supercritical, critical and subcritical is still relevant for KCMs, and 
established an even more precise classification. However, these results deal only with equilibrium dynamics. 
Until now, nothing had been shown on out-of-equilibrium 
KCMs with general update families, apart from a perturbative result in dimension 1 \cite{Cancrini_et_al2010}. 

In this article, we prove that for all supercritical update families, for any initial law $\nu_{q'}$, $q'\in]0,1]$, 
when $q$ is close enough to 1, the dynamics of the KCM converges to equilibrium with exponential speed. 
This result holds in dimension 2 and also in dimension 1 for a good definition of 
one-dimensional supercritical update families. It is the first non-perturbative result of convergence 
to equilibrium holding for a whole class of update families.

This result may help to gain a better understanding of the out-of-equilibrium behavior of supercritical KCMs. 
In particular, such results of convergence to equilibrium were key in proving ``shape theorems'' 
for specific one-dimensional constraints in \cite{Blondel2013,Ganguly_et_al,Blondel_et_al2018}.

\section{Notations and result}

Let $d \in \mathds{N}^*$. We denote by $\|.\|_\infty$ the $\ell^\infty$-norm on $\mathds{Z}^d$. 
For any set $S$, $|S|$ will denote the cardinal of~$S$. 

For any configuration $\eta \in \{0,1\}^{\mathds{Z}^d}$, for any $x\in \mathds{Z}^d$, 
we denote $\eta(x)$ the value of $\eta$ at $x$. Moreover, for any $S \subset \mathds{Z}^d$, 
we denote $\eta_S$ the restriction of $\eta$ to $S$, and $0_S$ (or just 0 when $S$ is clear from the context) 
the configuration on $\{0,1\}^S$ that contains only zeroes. 

We set an update family $\mathcal{U}=\{X_1,\dots,X_m\}$ with $m \in \mathds{N}^*$ and the $X_i$ 
finite nonempty subsets of $\mathds{Z}^d \setminus \{0\}$. 
To describe the classification of update families, we need the concept of \emph{stable directions}. 

\begin{definition}
For $u \in S^{d-1}$, we denote $\mathds{H}_u = \{x \in \mathds{R}^d \,|\, \langle x,u \rangle < 0 \}$ 
the half-space with boundary orthogonal to $u$. We say that $u$ is a \emph{stable direction} 
for the update family $\mathcal{U}$ if there does not exist $X \in \mathcal{U}$ such that 
$X \subset \mathds{H}_u$; otherwise $u$ is \emph{unstable}. We denote by $\mathcal{S}$ the set of stable directions. 
\end{definition}

\cite{Bollobas_et_al2015} gave a classification of two-dimensional update families into 
supercritical, critical or subcritical depending on their stable directions. 
Here is the generalization proposed for $d$-dimensional update families 
by \cite{Bollobas_et_al2017} (definition 9.1 therein), where for any $\mathcal{E} \subset S^{d-1}$, 
$\mathrm{int}(\mathcal{E})$ is the interior of $\mathcal{E}$ in the usual topology on $S^{d-1}$. 

\begin{definition}\label{def_universality_classes}
A $d$-dimensional update family $\mathcal{U}$ is 
\begin{itemize}
\item supercritical if there exists an open hemisphere $C \subset S^{d-1}$ that contains no stable direction; 
\item critical if every open hemisphere $C \subset S^{d-1}$ contains a stable direction, 
but there exists a hemisphere $C \subset S^{d-1}$ such that $\mathrm{int}(C \cap \mathcal{S}) = \emptyset$;
\item subcritical if $\mathrm{int}(C \cap \mathcal{S}) \neq \emptyset$ for every hemisphere $C \subset S^{d-1}$.
\end{itemize}
\end{definition}

Our result will be valid for supercritical update families. 

The KCM process with update family $\mathcal{U}$ can be constructed as follows. We set $q \in [0,1]$. 
Independently for all $x \in \mathds{Z}^d$, we define two independent Poisson point processes
$\mathcal{P}^0_x$ and $\mathcal{P}^1_x$ on $[0,+\infty[$, with respective rates $q$ and $1-q$. 
We call the elements of $\mathcal{P}^0_x \cup \mathcal{P}^1_x$ \emph{clock rings} 
and denote them by $t_{1,x} < t_{2,x} < \cdots$. The 
elements of $\mathcal{P}^0_x$ will be \emph{0-clock rings} and the elements of 
$\mathcal{P}^1_x$ will be \emph{1-clock rings}. 
For any intial configuration $\eta \in \{0,1\}^{\mathds{Z}^d}$, we construct the KCM as the 
continuous-time process $(\eta_t)_{t \in [0,+\infty[}$ on $\{0,1\}^{\mathds{Z}^d}$ defined thus: 
for any $x \in \mathds{Z}^d$, $\eta_t(x)=\eta_0(x)$ for $t \in [0,t_{1,x}[$, and for any 
$k \in \mathds{N}^*$, 
\begin{itemize}
\item if there exists $X \in \mathcal{U}$ such that $(\eta_{t_{k,x}^-})_{x+X}=0_{x+X}$, then 
$\eta_t(x)=\varepsilon$ for $t \in [t_{k,x},t_{k+1,x}[$, where $t_{x,k}$ is a 
$\varepsilon$-clock ring, $\varepsilon \in \{0,1\}$; 
\item if such an $X$ does not exist, $\eta_t(x)=\eta_{t_{k-1,x}}(x)$ for $t \in [t_{k,x},t_{k+1,x}[$.
\end{itemize}
In other words, sites try to update themselves to 0 when there is a 0-clock ring, which happens at rate $q$,
and to 1 when there is a 1-clock ring, which happens at rate $1-q$, 
but an update at $x$ is successful if and only if there exists an update rule $X$ such that 
all sites of $x+X$ are at zero. This construction is known as \emph{Harris graphical construction}. 
One can use the arguments in part 4.3 of \cite{Swart2017} to see that it is well-defined.
We denote by $\mathds{P}_\nu$ the law of $(\eta_t)_{t \in [0,+\infty[}$ when the initial configuration has law $\nu$.

For any $q' \in [0,1]$, we denote $\nu_{q'}$ the product $\mathrm{Bernoulli}(1-q')$ measure. 
Since the constraint at a site does not depend on the state of the site, it can be easily checked that $\nu_q$ 
satisfies the detailed balance with respect to the dynamics, hence is reversible and invariant. 
$\nu_q$ is called equilibrium measure of the dynamics. 

We will say that a function $f : \{0,1\}^{\mathds{Z}^d} \mapsto \mathds{R}$ is 
\emph{local} if its output depends only on the states of a finite set of sites, and we then denote 
$\|f\|_\infty = \sup_{\eta \in \{0,1\}^{\mathds{Z}^d}}|f(\eta)|$ its norm.

\begin{theorem}\label{thm_convergence}
If $d=1$ or 2, for any supercritical update family $\mathcal{U}$, 
for any $q' \in ]0,1]$, there exists $q_0=q_0(\mathcal{U},q') \in [0,1[$ 
such that for any $q \in [q_0,1]$, for any local function 
$f: \{0,1\}^{\mathds{Z}^d} \mapsto \mathds{R}$, 
there exist two constants $c=c(\mathcal{U},q')>0$ and $C=C(\mathcal{U},q',f)>0$ such that 
for any $t \in [0,+\infty[$,
\[
  \left| \mathds{E}_{\nu_{q'}} (f(\eta_t))-\nu_q(f) \right| \leq C e^{-ct}.
\]
\end{theorem}

\begin{remark}
We expect theorem \ref{thm_convergence} to hold also for $d \geq 3$. However, our proof 
relies on proposition \ref{prop_Bollobas}, which is easy for $d=1$ and was proven in \cite{Bollobas_et_al2015} 
for $d=2$, but for which there is no equivalent for $d \geq 3$. Such an equivalent would extend our result 
to $d \geq 3$.
\end{remark}

The remainder of this article is devoted to the proof of theorem \ref{thm_convergence}. 
The argument is based on the proof given in \cite{Mountford_FA1f} for the particular case of the 
Fredrickson-Andersen one-spin facilitated model, but brings in novel ideas 
in order to accommodate the much greater complexity of general supercritical models. 
From now on, we fix $d=1$ or 2 and $\mathcal{U}$ a supercritical update family in dimension $d$. 
We begin in section \ref{sec_dual_paths} by using the notion of 
dual paths to reduce the proof of theorem \ref{thm_convergence} 
to the simpler proof of proposition \ref{prop_all_paths_activated}. Then in section \ref{sec_codings} 
we use the concept of codings to simplify it further, reducing it to the proof of proposition \ref{prop_bound_single_coding}. 
In section \ref{sec_aux_proc} we introduce an auxiliary oriented percolation process, that we 
use in section \ref{sec_preuve_codings} to prove proposition 
\ref{prop_bound_single_coding} hence theorem \ref{thm_convergence}. 

\section{Dual paths}\label{sec_dual_paths}

In this section, we use the concept of \emph{dual paths} 
to reduce the proof of theorem \ref{thm_convergence} to the easier 
proof of proposition \ref{prop_all_paths_activated}.
Let $q,q' \in [0,1]$. We notice that the Harris graphical construction allows us to 
couple a process $(\eta_t)_{t \in [0,+\infty[}$ with initial law $\nu_{q'}$ and 
a process $(\tilde{\eta}_t)_{t \in [0,+\infty[}$ with initial law $\nu_q$ 
by using the same clock rings but different initial configurations (independent of the clock rings and of each other). 
We denote the joint law by $\mathds{P}_{q',q}$. We notice that since $\nu_q$ is an invariant measure for 
the dynamics, $\tilde{\eta}_t$ has law $\nu_q$ for all $t \in [0,+\infty[$. To prove theorem \ref{thm_convergence}, 
it is actually enough to show 

\begin{proposition}\label{prop_conv_1site}
For any $q' \in ]0,1]$, there exists $q_0=q_0(\mathcal{U},q') \in [0,1[$ 
such that for any $q \in [q_0,1]$, there exist 
two constants $c_1=c_1(\mathcal{U},q')>0$ and $C_1=C_1(\mathcal{U},q')>0$ such that 
for any $x \in \mathds{Z}^d$ and $t \in [0,+\infty[$, 
$\mathds{P}_{q',q}(\eta_t(x) \neq \tilde{\eta}_t(x)) \leq C_1 e^{-c_1t}$.
\end{proposition}

Indeed, if $f: \{0,1\}^{\mathds{Z}^d} \mapsto \mathds{R}$ is a local function 
depending of a finite set of sites $S$, 
\[
  \left| \mathds{E}_{\nu_{q'}} (f(\eta_t))-\nu_q(f) \right| 
  = \left| \mathds{E}_{q',q} (f(\eta_t))-\mathds{E}_{q',q}(f(\tilde{\eta}_t)) \right| 
\leq \mathds{E}_{q',q}(|f(\eta_t)-f(\tilde{\eta}_t)|) 
\]
\[
\leq 2\|f\|_\infty \mathds{P}_{q',q}((\eta_t)_S \neq (\tilde{\eta}_t)_S)
\leq 2\|f\|_\infty \sum_{x \in S} \mathds{P}_{q',q}(\eta_t(x) \neq \tilde{\eta}_t(x)).
\]

Therefore we will work on proving proposition \ref{prop_conv_1site}. 

In order to do that, we need to introduce dual paths. We define the \emph{range} $\rho$ of $\mathcal{U}$ by 
\[
\rho = \max \{\|x\|_\infty \,|\, x \in X, X \in \mathcal{U}\}.
\]
For any $x \in \mathds{Z}^d$, $t > 0$ and $0 \leq t' \leq t$, 
a dual path of length $t'$ starting at $(x,t)$ (see figure \ref{fig_dual_paths}) 
is a right-continuous path $(\Gamma(s))_{0 \leq s \leq t'}$ that starts at site $x$ at time $t$, 
goes backwards, is allowed to jump only when there is a clock ring, and only to a site within 
$\ell^\infty$-distance $\rho$. To write it rigorously, 
the path satisfies $\Gamma(0)=x$ and there exists a sequence 
of times $0=s_0 < s_1 < \cdots < s_{n}=t'$ satisfying the following properties: 
for all $0 \leq k \leq n-1$ and all $s \in [s_k,s_{k+1}[$, 
$\Gamma(s)=\Gamma(s_k)$, $\Gamma(s_n) = \Gamma(s_{n-1})$ and 
for all $0 \leq k < n-1$, $t-s_{k+1} \in \mathcal{P}^0_{\Gamma(s_k)} \cup 
\mathcal{P}^1_{\Gamma(s_k)}$ and $\|\Gamma(s_{k+1}) - \Gamma(s_k)\|_\infty \leq \rho$. 

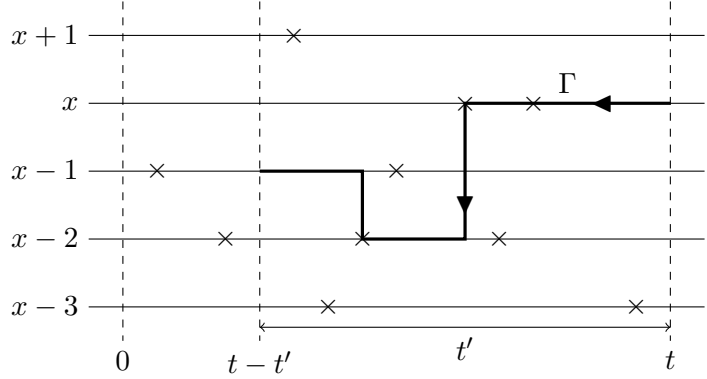
\begin{figure}
\begin{tikzpicture}[scale=0.9]
\draw (0,0)--(9,0) ;
\draw (0,1)--(9,1) ;
\draw (0,2)--(9,2) ;
\draw (0,3)--(9,3) ;
\draw (0,4)--(9,4) ;
\draw (0,0) node[left]{$x-3$} ;
\draw (0,1) node[left]{$x-2$} ;
\draw (0,2) node[left]{$x-1$} ;
\draw (0,3) node[left]{$x$} ;
\draw (0,4) node[left]{$x+1$} ;
\draw[dashed] (0.5,4.5)--(0.5,-0.5) node[below] {$0$} ;
\draw[dashed] (2.5,4.5)--(2.5,-0.5) node[below] {$t-t'$} ;
\draw[dashed] (8.5,4.5)--(8.5,-0.5) node[below] {$t$} ;
\draw[<->] (2.5,-0.3)--(8.5,-0.3) node[midway,below] {$t'$} ;
\draw (1,2) node {$\times$} ;
\draw (2,1) node {$\times$} ;
\draw (3,4) node {$\times$} ;
\draw (3.5,0) node {$\times$} ;
\draw (4,1) node {$\times$} ;
\draw (4.5,2) node {$\times$} ;
\draw (5.5,3) node {$\times$} ;
\draw (6,1) node {$\times$} ;
\draw (6.5,3) node {$\times$} ;
\draw (8,0) node {$\times$} ;
\draw[very thick] (8.5,3)--(5.5,3)--(5.5,1)--(4,1)--(4,2)--(2.5,2) ;
\draw (5.5,1.5) node{$\blacktriangledown$};
\draw (7.5,3) node{$\blacktriangleleft$};
\draw (7,3) node[above]{$\Gamma$};
\end{tikzpicture}
\caption{Illustration of a dual path $\Gamma$ of length $t'$ starting at $(x,t)$ for $d=1$ and $\rho=2$. 
Each horizontal line represents the timeline 
of a site of $\mathds{Z}$, the $\times$ representing the clock rings. $\Gamma$ is the thick polygonal line; 
it starts at $t$ and ends at $t-t'$. It can jump only when there is a clock ring, and never at a distance 
greater than $\rho=2$.}
\label{fig_dual_paths}
\end{figure}

We denote $\mathcal{D}(x,t,t')$ the (random) set of all dual paths of length $t'$ starting from $(x,t)$.
A dual path $\Gamma \in \mathcal{D}(x,t,t')$ is 
called an \emph{activated path} if it ``encounters a point at which both processes are at 0'', 
i.e. if there exists $s \in [0,t']$ such that $\eta_{t-s}(\Gamma(s))=\tilde{\eta}_{t-s}(\Gamma(s))=0$.
The set of all activated paths in $\mathcal{D}(x,t,t')$ is called $\mathcal{A}(x,t,t')$. We have the 

\begin{lemma}\label{lemma_activated_paths}
For any $x \in \mathds{Z}^d$ and $t > 0$, if $\eta_t(x) \neq \tilde{\eta}_t(x)$, then 
for all $0 \leq t' \leq t$, $\mathcal{A}(x,t,t')\neq\mathcal{D}(x,t,t')$.
\end{lemma}

\begin{proof}[Sketch of proof.] The proof is the same as for lemma 1 of \cite{Mountford_FA1f}, apart from the 
fact that if the path is at $y$, it does not necessarily jump to a neighbor of $y$, but to an element 
of $y+X$, $X \in \mathcal{U}$. The idea of the proof is to start a dual path at $(x,t)$, 
where the two processes disagree, and, staying at $x$, to go backwards in time until the 
processes agree at $x$. At this time, there was an update at $x$ in one process 
but not in the other, hence an update rule $x+X$ that was full of zeroes in one process but not in the other, 
thus a site at distance at most $\rho$ of $x$ at which the two processes disagree. We 
jump to this site and continue to go backwards. This construction yields a dual path along which the 
two processes disagree, hence they can not be both at zero, so the path is not activated.
\end{proof}

Lemma \ref{lemma_activated_paths} implies that to prove proposition \ref{prop_conv_1site} 
hence theorem \ref{thm_convergence}, it is enough to prove

\begin{proposition}\label{prop_all_paths_activated}
For any $q' \in ]0,1]$, there exists $q_0=q_0(\mathcal{U},q') \in [0,1[$ 
such that for any $q \in [q_0,1]$, there exist 
two constants $c_2=c_2(\mathcal{U},q')>0$ and $C_2=C_2(\mathcal{U},q')>0$ such that 
for any $x \in \mathds{Z}^d$, $t \in [0,+\infty[$, there exists $0 \leq t' \leq t$ 
such that $\mathds{P}_{q',q}(\mathcal{A}(x,t,t') \neq\mathcal{D}(x,t,t')) \leq C_2 e^{-c_2t}$.
\end{proposition}

The remainder of the article will be devoted to the proof of proposition \ref{prop_all_paths_activated}. 

\section{Codings}\label{sec_codings}

This section is devoted to the reduction of the proof of proposition \ref{prop_all_paths_activated} 
(hence of theorem \ref{thm_convergence}) to the simpler proof of proposition \ref{prop_bound_single_coding}, 
via the use of \emph{codings}. The idea is the following: in order to prove 
proposition \ref{prop_all_paths_activated}, it is enough to show that along each dual path, the two processes 
are at zero at one of the discrete times $0$, $K$, $2K$, etc. hence we only need to consider the positions 
of the path at these times, which will make up the coding of the path.
Let $K \geq 2$ and $t \geq K$. A coding is a sequence 
$(y_k)_{k \in \{0,\dots,\lfloor\frac{t}{K^2}\rfloor\}}$ of sites in $\mathds{Z}^d$. 
Moreover, for $x \in \mathds{Z}^d$ and $\Gamma \in \mathcal{D}(x,t,\frac{t}{K})$, 
the coding $\bar{\Gamma}$ of $\Gamma$ is the sequence 
$\{\Gamma(kK)\}_{k\in\{0,\dots,\lfloor\frac{t}{K^2}\rfloor\}}$. 
If $\gamma = (y_k)_{k \in \{0,\dots,\lfloor\frac{t}{K^2}\rfloor\}}$ is a coding, we define the event 
$G(\gamma) = \left\{\exists k\in\left\{0,\dots,\left\lfloor\frac{t}{K^2}\right\rfloor\right\}, 
\eta_{t-kK}(y_k) = \tilde{\eta}_{t-kK}(y_k)=0\right\}$. 
If $G(\bar{\Gamma})$ is satisfied, $\Gamma$ is an activated path. 

Therefore, to prove proposition \ref{prop_all_paths_activated} hence theorem \ref{thm_convergence}, 
it is enough to prove 

\begin{proposition}\label{prop_paths_to_codings}
For any $q' \in ]0,1]$, there exists $q_0=q_0(\mathcal{U},q') \in [0,1[$ 
such that for any $q \in [q_0,1]$, there exist 
two constants $c_3=c_3(\mathcal{U},q')>0$ and $C_3=C_3(\mathcal{U},q')>0$ 
and a constant $K = K(\mathcal{U},q') \geq 2$ such that 
for any $x \in \mathds{Z}^d$ and $t \geq 2K^2$, $\mathds{P}_{q',q}(\exists \Gamma \in 
\mathcal{D}(x,t,\frac{t}{K}),G(\bar{\Gamma})^c) \leq C_3 e^{-c_3t}$.
\end{proposition}

Proposition \ref{prop_paths_to_codings} holds only for $t$ greater than a constant, but 
this is enough, since we only have to enlarge $C_3$ to obtain a bound valid for all $t$. 

In order to prove proposition \ref{prop_paths_to_codings}, we will define a set $C_K^N(x,t)$ 
of ``reasonable codings'' and prove that the probability that there exists a dual path whose coding is not 
in $C_K^N(x,t)$ decays exponentially in $t$ (lemma \ref{lemma_long_dual_paths}). 
Then we will count the number of codings in $C_K^N(x,t)$ 
(lemma \ref{lemma_number_codings}). Therefore it will be enough to give a bound on 
$\mathds{P}_{q',q}(G(\gamma)^c)$ for any $\gamma \in C_K^N(x,t)$ to prove proposition 
\ref{prop_paths_to_codings} hence theorem \ref{thm_convergence}. Such a bound is stated 
in proposition \ref{prop_bound_single_coding} and will be proven in section \ref{sec_preuve_codings}.

For any constant $N > 0$, for any $K \geq 2$, $x \in \mathds{Z}^d$ and $t \geq K$, 
the set $C_K^N(x,t)$ of ``reasonable codings'' is defined as the set of 
$(y_{j_1+\cdots+j_k})_{k \in \{0,\dots,\lfloor\frac{t}{K^2}\rfloor\}}$ where 
$(y_{i})_{i \in \{0,\dots,I\}}$ is a sequence of sites satisfiying $y_0=x$, $I \leq \frac{Nt}{K}$ and 
$\|y_{i+1}-y_i\|_\infty \leq \rho$ for all $i \in \{0,\dots,I-1\}$ and where 
$j_1,\dots,j_{\lfloor\frac{t}{K^2}\rfloor} \in \mathds{N}$ satisfy $j_1+\cdots+j_{\lfloor\frac{t}{K^2}\rfloor} \leq I$. 
We can now state lemmas \ref{lemma_long_dual_paths} and \ref{lemma_number_codings}, as well 
as proposition \ref{prop_bound_single_coding}. These statements together prove 
proposition \ref{prop_paths_to_codings}.

\begin{lemma}\label{lemma_long_dual_paths}
For any $q' \in [0,1]$, there exists $N=N(\mathcal{U}) > 0$ 
such that for any $K \geq 2$, $q \in [0,1]$, there exists a constant 
$\check{c}=\check{c}(\mathcal{U},K)>0$ such that for all $x \in \mathds{Z}^d$ and $t \geq K$, 
$\mathds{P}_{q',q}(\exists \Gamma \in \mathcal{D}(x,t,\frac{t}{K}), 
\bar{\Gamma} \not \in C_K^N(x,t)) \leq e^{-\check{c}t}$.
\end{lemma}

In the following, $N$ will always be the $N$ given by lemma \ref{lemma_long_dual_paths}. 

\begin{lemma}\label{lemma_number_codings}
There exist constants $\lambda > 0$ and $\beta = \beta(\mathcal{U}) > 0$ such that for any $K \geq 2$, 
$x \in \mathds{Z}^d$ and $t \geq 2K^2$, $|C_K^N(x,t)| \leq \lambda (\beta K)^{(d+1)\frac{t}{K^2}}$.
\end{lemma}

\begin{proposition}\label{prop_bound_single_coding}
For any $q' \in [0,1]$, there exists a constant $K_0=K_0(\mathcal{U}) \geq 2$ such that 
for any $K \geq K_0$, there exists $q_K \in [0,1[$ such that for any $q \in [q_K,1]$, 
there exist two constants $c_4=c_4(\mathcal{U},q')>0$ and $C_4=C_4(\mathcal{U},K)>0$ such that 
for any $x \in \mathds{Z}^d$, $t \geq K$ and $\gamma \in C_K^N(x,t)$, 
$\mathds{P}_{q',q}(G(\gamma)^c) \leq C_4 e^{-c_4\frac{t}{K}}$.
\end{proposition}

We are now going to prove lemmas \ref{lemma_long_dual_paths} and \ref{lemma_number_codings}. 
After that, it will suffice to prove proposition \ref{prop_bound_single_coding} to prove 
theorem \ref{thm_convergence}.

\begin{proof}[Sketch of proof of lemma \ref{lemma_long_dual_paths}.]
This can be proven with the argument of the lemma 5 of \cite{Mountford_FA1f}; the idea is that if there exists 
$\Gamma \in \mathcal{D}(x,t,\frac{t}{K})$ with $\bar{\Gamma} \not \in C_K^N(x,t)$, 
there are so many clock rings that the probability becomes very small. Indeed, let us say 
$\Gamma$ visits the sites $y_0=x,y_1,\dots,y_{j_1}$ in the time interval $[0,K]$, then the sites 
$y_{j_1},\dots,y_{j_1+j_2}$  in the time interval $[K,2K]$, etc. until the sites 
$y_{j_1+\cdots+j_{\lfloor\frac{t}{K^2}\rfloor}},\dots,y_{j_1+\cdots+j_{\lfloor\frac{t}{K^2}\rfloor+1}}$ 
in the time interval $[\lfloor\frac{t}{K^2}\rfloor K,(\lfloor\frac{t}{K^2}\rfloor+1)K]$. 
Then the coding of $\Gamma$ is $\bar{\Gamma} = (y_{j_1+\cdots+j_k})_{k \in \{0,\dots,\lfloor\frac{t}{K^2}\rfloor\}}$, 
hence $\bar{\Gamma} \not \in C_K^N(x,t)$ implies $j_1+\cdots+j_{\lfloor\frac{t}{K^2}\rfloor+1} > \frac{Nt}{K}$. 
It yields that $\Gamma$ visits more than $\frac{Nt}{K}$ sites in 
a time $\frac{t}{K}$, and there must be successive clock rings at these sites. The 
proof of lemma 5 of \cite{Mountford_FA1f} yields that we can choose $N$ large enough 
depending on $\rho$, hence on $\mathcal{U}$, such that the probability of this event is 
at most $e^{-\check{c}t}$ with $\check{c}=\check{c}(\mathcal{U},N,K) = \check{c}(\mathcal{U},K) > 0$.
\end{proof}

To prove lemma \ref{lemma_number_codings}, we need the following classical combinatorial result, 
which will also be used in the proof of lemma \ref{lemma_percolation_structure}. 

\begin{lemma}\label{lemma_binomial_coeffs}
For any $I,J \in \mathds{N}$, $\binom{I}{I} + \binom{I+1}{I} + \cdots + \binom{I + J}{I} = \binom{I+J+1}{I+1}$.
Moreover, for any $I,J \in \mathds{N}$, $|\{(j_1,\dots,j_I)\in\mathds{N}^I \,|\, j_1+\cdots+j_I=J\}| = 
\binom{I+J-1}{I-1}$.
\end{lemma}

The proof of the first part of lemma \ref{lemma_binomial_coeffs} can be found 
just before the section 2 of \cite{Jones_1996} and the proof of the second part in 
section 1.2 of \cite{Stanley_enucomb} (weak compositions).

\begin{proof}[Proof of lemma \ref{lemma_number_codings}.]
Let $K \geq 2$, $x \in \mathds{Z}^d$ and $t \geq 2K^2$.
By definition, elements of $C_K^N(x,t)$ have the form $(y_{j_1+\cdots+j_k})_{k \in \{0,\dots,\lfloor\frac{t}{K^2}\rfloor\}}$ 
with $(y_{i})_{i \in \{0,\dots,I\}}$ satisfiying $y_0=x$, $I \leq \frac{Nt}{K}$ and 
$\|y_{i+1}-y_i\|_\infty \leq \rho$ for all $i \in \{0,\dots,I-1\}$, and with 
  $j_1,\dots,j_{\lfloor\frac{t}{K^2}\rfloor} \in \mathds{N}$ satisfying $j_1+\cdots+j_{\lfloor\frac{t}{K^2}\rfloor} \leq I$. 
Therefore, to count the number of elements of $C_K^N(x,t)$, it is enough to count the number 
of possible $(j_k)_{k \in \{1,\dots,\lfloor\frac{t}{K^2}\rfloor\}}$ and the 
number of possible $(y_{j_1+\cdots+j_k})_{k \in \{0,\dots,\lfloor\frac{t}{K^2}\rfloor\}}$ 
given $(j_k)_{k \in \{1,\dots,\lfloor\frac{t}{K^2}\rfloor\}}$.

We begin by counting the number of possible $(j_k)_{k \in \{1,\dots,\lfloor\frac{t}{K^2}\rfloor\}}$. 
We have $j_1+\cdots+j_{\lfloor\frac{t}{K^2}\rfloor} \leq \frac{Nt}{K}$. Moreover, 
by the second part of lemma \ref{lemma_binomial_coeffs}, 
for any integer $0 \leq J \leq \frac{Nt}{K}$, the number of possible sequences of integers 
$(j_k)_{k \in \{1,\dots,\lfloor\frac{t}{K^2}\rfloor\}}$ such that 
$j_1+\cdots+j_{\lfloor\frac{t}{K^2}\rfloor} =J$ is at most $\binom{\lfloor\frac{t}{K^2}\rfloor+J-1}
{\lfloor\frac{t}{K^2}\rfloor-1}$, hence the number of possible 
$(j_k)_{k \in \{1,\dots,\lfloor\frac{t}{K^2}\rfloor\}}$ is at most 
$\sum_{J=0}^{\lfloor\frac{Nt}{K}\rfloor}\binom{\lfloor\frac{t}{K^2}\rfloor+J-1}
{\lfloor\frac{t}{K^2}\rfloor-1} = \binom{\lfloor\frac{t}{K^2}\rfloor+\lfloor\frac{Nt}{K}\rfloor}
{\lfloor\frac{t}{K^2}\rfloor}$ by the first part of lemma \ref{lemma_binomial_coeffs}.
Furthermore $\binom{\lfloor\frac{t}{K^2}\rfloor+\lfloor\frac{Nt}{K}\rfloor}{\lfloor\frac{t}{K^2}\rfloor} 
\leq \frac{(\lfloor\frac{t}{K^2}\rfloor+\lfloor\frac{Nt}{K}\rfloor)^{\lfloor\frac{t}{K^2}\rfloor}}
{(\lfloor\frac{t}{K^2}\rfloor)!} \leq 
\lambda\left(\frac{e(\lfloor\frac{t}{K^2}\rfloor+\lfloor\frac{Nt}{K}\rfloor)}
{\lfloor\frac{t}{K^2}\rfloor}\right)^{\lfloor\frac{t}{K^2}\rfloor} \leq 
\lambda\left(e+e\frac{\lfloor\frac{Nt}{K}\rfloor}{\lfloor\frac{t}{K^2}\rfloor}\right)^{\frac{t}{K^2}}$ 
by the Stirling formula, where $\lambda > 0$ is a constant.
In addition, since $t \geq 2K^2$, $\lfloor\frac{t}{K^2}\rfloor\geq \frac{t}{2K^2}$, hence the 
number of possible $(j_k)_{k \in \{1,\dots,\lfloor\frac{t}{K^2}\rfloor\}}$ is at most 
$\lambda\left(e+e\frac{Nt}{K}\frac{2K^2}{t}\right)^{\frac{t}{K^2}}
= \lambda\left(e+2eKN\right)^{\frac{t}{K^2}} \leq \lambda (3eKN)^{\frac{t}{K^2}}$ as $K \geq 2$ and $N$ is large.

We now fix a sequence $(j_k)_{k \in \{1,\dots,\lfloor\frac{t}{K^2}\rfloor\}}$ and 
count the possible $(y_{j_1+\cdots+j_k})_{k \in \{0,\dots,\lfloor\frac{t}{K^2}\rfloor\}}$. 
We know that $y_0=x$. Moreover, for all $i \in \{0,\dots,j_1+\cdots+j_{\lfloor\frac{t}{K^2}\rfloor}-1\}$, 
$\|y_{i+1}-y_i\|_\infty \leq \rho$, hence for each $k \in \{0,\dots,\lfloor\frac{t}{K^2}\rfloor-1\}$, we have 
$\|y_{j_1+\cdots+j_{k+1}}-y_{j_1+\cdots+j_k}\|_\infty \leq \rho j_{k+1}$, so there are 
at most $(2\rho j_{k+1}+1)^d$ choices for $y_{j_1+\cdots+j_{k+1}}$ given $y_{j_1+\cdots+j_{k}}$. 
Therefore the number of choices for 
$(y_{j_1+\cdots+j_k})_{k \in \{0,\dots,\lfloor\frac{t}{K^2}\rfloor\}}$ is at most 
$\prod_{k=1}^{\lfloor\frac{t}{K^2}\rfloor}(2\rho j_{k}+1)^d$. Moreover, for any 
$n \in \mathds{N}^*$ and any positive $x_1,\dots,x_n$, we have $x_1\dots x_n \leq (\frac{x_1+\cdots+x_n}{n})^n$, 
therefore the number of choices is bounded by 
\[
\left(\frac{\sum_{k=1}^{\lfloor\frac{t}{K^2}\rfloor}(2\rho j_{k}+1)}
{\lfloor\frac{t}{K^2}\rfloor}\right)^{d\lfloor\frac{t}{K^2}\rfloor}
= \left(\frac{2\rho\sum_{k=1}^{\lfloor\frac{t}{K^2}\rfloor}j_{k}+\lfloor\frac{t}{K^2}\rfloor}
{\lfloor\frac{t}{K^2}\rfloor}\right)^{d\lfloor\frac{t}{K^2}\rfloor} 
\leq \left(\frac{2\rho\frac{Nt}{K}+\lfloor\frac{t}{K^2}\rfloor}
{\lfloor\frac{t}{K^2}\rfloor}\right)^{d\frac{t}{K^2}}
\]
since $\sum_{k=1}^{\lfloor\frac{t}{K^2}\rfloor}j_{k} \leq \frac{Nt}{K}$. As $t \geq 2K^2$, 
$\lfloor\frac{t}{K^2}\rfloor\geq \frac{t}{2K^2}$, thus the number of 
choices for $(y_{j_1+\cdots+j_k})_{k \in \{0,\dots,\lfloor\frac{t}{K^2}\rfloor\}}$ 
given $(j_k)_{k \in \{1,\dots,\lfloor\frac{t}{K^2}\rfloor\}}$ 
is bounded by $\left(2\rho\frac{Nt}{K}\frac{2K^2}{t}+1\right)^{d\frac{t}{K^2}} = 
(4\rho NK+1)^{d\frac{t}{K^2}} \leq (5\rho NK)^{d\frac{t}{K^2}}$.
\end{proof}

\section{An auxiliary process}\label{sec_aux_proc}

In order to prove proposition \ref{prop_bound_single_coding}, we need to find a mechanism for the 
zeroes to spread in the KCM process; this mechanism uses novel ideas 
to deal with the complexity of general supercritical models. We begin in section 
\ref{subsec_bootstrap_result} by using the bootstrap percolation 
results of \cite{Bollobas_et_al2015} to find a mechanism allowing the zeroes to spread locally 
(proposition \ref{prop_Bollobas}). Then we use it in section \ref{subsec_def_aux_process} 
to define an auxiliary oriented percolation process which guarantees that if certain 
conditions are met, the KCM process is at zero at a given time (proposition \ref{prop_transfer_zeroes2}). 
Finally, in section \ref{subsec_prop_aux_process} we prove some properties of this auxiliary process 
that we will use in section \ref{sec_preuve_codings}. 

\subsection{Local spread of zeroes}\label{subsec_bootstrap_result}

This is the place where we need the supercriticality of $\mathcal{U}$. Indeed, since $\mathcal{U}$ is supercritical, 
the results of \cite{Bollobas_et_al2015} yield the following proposition (see figure \ref{fig_prop_Bollobas}): 

\begin{proposition}[\cite{Bollobas_et_al2015}]\label{prop_Bollobas}
For $d=1$ or $2$, there exists $u \in S^{d-1}$, a rectangle $R$ of the following form:
\begin{itemize}
  \item if $d=1$, $R = [0,a_1u[ \cap \mathds{Z}$ with $a_1 u \in \mathds{Z}$; 
  \item if $d=2$, $R = ([0,a_1[u+[0,a_2]u^\perp) \cap \mathds{Z}^2$ with $a_1 u \in \mathds{Z}^2$, 
  where $u^\perp$ is a vector orthogonal to $u$,
\end{itemize}
and a sequence of sites $(x_i)_{1 \leq i \leq m}$ in 
$(a_1u+R) \cup (2a_1u+R)$ such that if the sites of $R$ are at zero and 
there are successive 0-clock rings at $x_1,x_2,\dots,x_m$ while there is no 1-clock ring 
in $R \cup \{x_1,\dots,x_m\}$, the sites of $a_1u+R$ are at zero afterwards.
\end{proposition}

\begin{figure}
\parbox{0.45\textwidth}{
\begin{center}
\begin{tikzpicture}[scale=0.44]
\draw (-1,0)--(13,0);
\draw [dashed] (-2,0)--(-1,0);
\draw [dashed] (13,0)--(14,0);
\draw (-2,0) node[left]{$\mathds{Z}$} ;
\foreach \i in {-1,0,...,13} \draw (\i,0.2)--(\i,-0.2) ;
\draw (0,0) node [below] {$0$} ;
\draw [ultra thick] (0,0.2)--(0,-0.2) ;
\draw (4,0) node [below] {$a_1 u$} ;
\draw [ultra thick] (4,0.2)--(4,-0.2) ;
\draw (8,0) node [below] {$2a_1 u$} ;
\draw [ultra thick] (8,0.2)--(8,-0.2) ;
\draw (12,0) node [below] {$3a_1 u$} ;
\draw [ultra thick] (12,0.2)--(12,-0.2) ;
\draw[decorate,decoration={brace}] (-0.5,0.2)--(3.5,0.2) node[midway,above]{$R$};
\draw[decorate,decoration={brace}] (3.5,0.2)--(7.5,0.2) node[midway,above]{$a_1u+R$};
\draw[decorate,decoration={brace}] (7.5,0.2)--(11.5,0.2) node[midway,above]{$2a_1u+R$};
\foreach \i in {4,5,...,9} \draw (\i,0) node{$\ast$};
\end{tikzpicture}

$d=1$
\end{center}
}
\parbox{0.45\textwidth}{
\begin{center}
\begin{tikzpicture}[scale=0.45]
\draw [very thin, gray] (-3,-1) grid (10,12);
\draw (0,0) node{$\times$} node[below,fill=white] {$0$};
\foreach \i in {-3,-2,...,10} \draw[very thin, gray, dashed] (\i,-1.5)--(\i,-1);
\foreach \i in {-3,-2,...,10} \draw[very thin, gray, dashed] (\i,12)--(\i,12.5);
\foreach \i in {-1,0,...,12} \draw[very thin, gray, dashed] (-3.5,\i)--(-3,\i);
\foreach \i in {-1,0,...,12} \draw[very thin, gray, dashed] (10,\i)--(10.5,\i);
\draw (-3.5,11) node [left] {$\mathds{Z}^2$} ;
\draw (0,0)--(3,3)--(1,5)--(-2,2)--cycle ;
\draw (3,3)--(1,5)--(4,8)--(6,6)--cycle ;
\draw (6,6)--(4,8)--(7,11)--(9,9)--cycle ;
\draw (-2,2)--(1,5) node [midway,above left,fill=white] {$R$};
\draw (1,5)--(4,8) node [midway,above left,fill=white] {$a_1u+R$};
\draw (4,8)--(7,11) node [midway,above left,fill=white] {$2a_1u+R$};
\draw [>=stealth,<->] (-0.2,-0.2)--(-2.2,1.8) node [midway,below left,fill=white] {$a_2$} ;
\draw [>=stealth,<->] (0.2,-0.2)--(3.2,2.8) node [midway,below right,fill=white] {$a_1$} ;
\draw [>=stealth,->] (7,0)--(7.7,0.7) node [midway,below right,fill=white] {$u$};
\draw (6,0) node [fill=white] {$u^\perp$};
\draw [>=stealth,->] (7,0)--(6.3,0.7);
\draw (1,5) node{$\ast$} ;
\draw (2,4) node{$\ast$} ;
\draw (3,3) node{$\ast$} ;
\draw (2,5) node{$\ast$} ;
\draw (3,4) node{$\ast$} ;
\draw (2,6) node{$\ast$} ;
\draw (3,5) node{$\ast$} ;
\draw (4,4) node{$\ast$} ;
\draw (3,6) node{$\ast$} ;
\draw (4,5) node{$\ast$} ;
\draw (3,7) node{$\ast$} ;
\draw (4,6) node{$\ast$} ;
\draw (5,5) node{$\ast$} ;
\draw (4,7) node{$\ast$} ;
\draw (5,6) node{$\ast$} ;
\draw (4,8) node{$\ast$} ;
\draw (5,7) node{$\ast$} ;
\draw (5,8) node{$\ast$} ;
\draw (6,7) node{$\ast$} ;
\draw (6,8) node{$\ast$} ;
\end{tikzpicture}

$d=2$
\end{center}
}

\caption{Illustration of proposition \ref{prop_Bollobas}. The $\ast$ represent the sites 
$x_1,\dots,x_m$. If we start with the sites of $R$ at zero and there are successive 0-clock rings at $x_1,\dots,x_m$ 
while there is no 1-clock ring in $R \cup \{x_1,\dots,x_m\}$, these clock rings will put $x_1,\dots,x_m$ at 
zero, hence the sites of $a_1u+R$ will be put at zero.}
\label{fig_prop_Bollobas}
\end{figure}
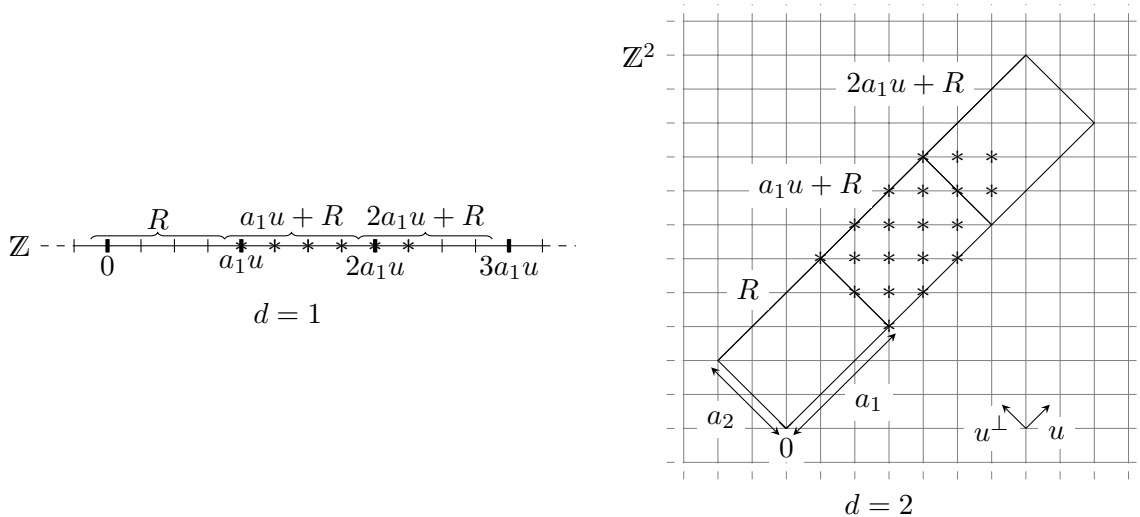

\begin{remark}
For $d \geq 3$, we expect a similar proposition to hold, maybe with $R = [0,a_1[u + \bar{R}$, 
$\bar{R}$ contained in the hyperplane orthogonal to $u$, but we can not prove it because an equivalent of the 
construction of \cite{Bollobas_et_al2015} is not available yet. 
Proving such a construction would be enough to extend our result to any dimension.
\end{remark}

\begin{proof}[Proof of proposition \ref{prop_Bollobas}.]
We begin with the case $d=1$. Since $\mathcal{U}$ is supercritical there exists $u$ 
an unstable direction. Without loss of generality 
we can say that $u = 1$, therefore there exists an update rule $X$ contained in $-\mathds{N}^*$. 
This yields the mechanism illustrated by figure \ref{fig_preuve_prop_Bollobas}(a): if 
$R = \{0,\dots,\ell\}$ is sufficiently large and full of zeroes, $(\ell+1)+X$ is full of zeroes, hence 
if the site $\ell+1$ receives a 0-clock ring, this clock ring puts it at zero. 
Then $(\ell+2)+X$ is full of zeroes, thus if $\ell+2$ receives a 0-clock ring, 
this clock ring puts it at zero. In the same way, if the sites $\ell+3,\dots,2\ell+1$ receive 
successive 0-clock rings, these clock rings will put them 
successively at zero, therefore $\{\ell+1,\dots,2\ell+1\} = (\ell+1) +R$ will be at zero. This yields the result 
with $a_1 = \ell+1$ and $(x_i)_{1 \leq i \leq m} = \ell+1,\ell+2,\dots,2\ell+1$.

We now consider the case $d = 2$. Since $\mathcal{U}$ is supercritical, there exists a semicircle in $S^{1}$ 
that contains no stable direction; we call $u$ its middle. The results of section 5 of \cite{Bollobas_et_al2015} 
(see in particular figure 5 and lemma 5.5 therein) prove that there exists a set of sites, called a 
\emph{droplet}, such that in the bootstrap percolation dynamics, if we start with 
all the sites of the droplet infected, other sites in the direction $u$ can be infected, 
creating a bigger infected droplet with the same shape (see figure \ref{fig_preuve_prop_Bollobas}(b)). 
We can enlarge this droplet into a rectangle $R = [0,a_1[u+[0,a_2]u^\perp$ 
as in figure \ref{fig_preuve_prop_Bollobas}(c); furthermore $u$ can be chosen 
rational\footnote{Indeed, theorem 1.10 of \cite{Bollobas_et_al2015} 
states that the set of stable directions is a finite union of closed intervals 
with rational endpoints, hence the semicircle containing no stable direction can be chosen with rational endpoints.}, 
hence we may enlarge $R$ enough so that $a_1 u \in \mathds{Z}^2$.  
Now, since $R$ contains the original droplet, if $R$ is infected the infection can grow from the 
droplet into a droplet big enough to contain $a_1 u + R$ while staying 
in $R \cup (a_1u+R) \cup (2a_1u+R)$ (see figure \ref{fig_preuve_prop_Bollobas}(c)). 
We call $x_1,\dots,x_m$ the sites that are successively infected during this growth 
(sites infected at the same time are ordered arbitrarily).
Since $x_1$ is the first site infected by the bootstrap percolation dynamics starting with the sites of $R$ 
infected, there exists an update rule $X$ such that $x_1+X \subset R$, therefore if 
the KCM dynamics starts with all the sites of $R$ at zero and there is a 0-clock ring at $x_1$, 
this clock ring sets $x_1$ to zero. Then, if there is a 0-clock ring at $x_2$, it will set $x_2$ to zero 
for the same reason, and successive 0-clock rings at $x_3,\dots,x_m$ will set them successively to 0, 
which puts $a_1u+R$ at zero.
\begin{figure}
\parbox{0.4\textwidth}{
\begin{center}
\begin{tikzpicture}[scale=0.35]
  \foreach \i in {-2,-1,...,13} \draw (\i,0) node{$\circ$};
  \foreach \i in {0,1,...,5} \draw (\i,0) node{$\bullet$};
  \draw [dotted] (-3.5,0)--(-2.5,0);
  \draw [dotted] (13.5,0)--(14.5,0);
  \draw (0,0) node[below] {$0$};
  \draw (5,0) node[below] {$\ell$};
  \draw[decorate,decoration={brace}] (-0.5,1.8)--(5.5,1.8) node[midway,above]{$R$};
  \draw[dotted] (-0.5,1.8)--(-0.5,-0.2);
  \draw[dotted] (5.5,1.8)--(5.5,-0.2);
  \draw (0.5,0.4)--(3.5,0.4)--(3.5,-0.4)--(0.5,-0.4)--cycle;
  \draw (2,0.2) node [above] {$(\ell+1)+X$};
  \draw (6,0) circle (0.4) node[above right]{$\ell+1$};
  \draw (5.5,-1.6) node{$\downarrow$} ;
  \foreach \i in {-2,-1,...,13} \draw (\i,-4) node{$\circ$};
  \foreach \i in {0,1,...,6} \draw (\i,-4) node{$\bullet$};
  \draw [dotted] (-3.5,-4)--(-2.5,-4);
  \draw [dotted] (13.5,-4)--(14.5,-4);
  \draw (0,-4) node[below] {$0$};
  \draw (5,-4) node[below] {$\ell$};
  \draw (1.5,-3.6)--(4.5,-3.6)--(4.5,-4.4)--(1.5,-4.4)--cycle;
  \draw (3,-3.8) node [above] {$(\ell+2)+X$};
  \draw (7,-4) circle (0.4) node[above right]{$\ell+2$};
  \draw (5.5,-6) node{$\downarrow$} ;
  \foreach \i in {-2,-1,...,13} \draw (\i,-7) node{$\circ$};
  \foreach \i in {0,1,...,7} \draw (\i,-7) node{$\bullet$};
  \draw [dotted] (-3.5,-7)--(-2.5,-7);
  \draw [dotted] (13.5,-7)--(14.5,-7);
  \draw (0,-7) node[below] {$0$};
  \draw (5,-7) node[below] {$\ell$};
  \draw (5.5,-9) node{$\downarrow$} ;
  \draw (5.5,-10) node{$\dots$} ;
  \draw (5.5,-11) node{$\downarrow$} ;
  \foreach \i in {-2,-1,...,13} \draw (\i,-13.5) node{$\circ$};
  \foreach \i in {0,1,...,11} \draw (\i,-13.5) node{$\bullet$};
  \draw [dotted] (-3.5,-13.5)--(-2.5,-13.5);
  \draw [dotted] (13.5,-13.5)--(14.5,-13.5);
  \draw (0,-13.5) node[below] {$0$};
  \draw (5,-13.5) node[below] {$\ell$};
  \draw (11,-13.5) node[below] {$2\ell+1$};
  \draw[decorate,decoration={brace}] (-0.5,-13.1)--(5.5,-13.1) node[midway,above]{$R$};
  \draw[decorate,decoration={brace}] (5.5,-13.1)--(11.5,-13.1) node[midway,above]{$(\ell+1)+R$};
\end{tikzpicture}

(a)
\end{center}
}
\parbox{0.23\textwidth}{
\begin{center}
\begin{tikzpicture}[scale=0.3]
  \draw (0,0)--(-3,3)--(-2,4)--(0,5)--(3,4)--(2,2)--cycle;
  \draw[dashed] (0,0)--(-3,3)--(0,6)--(2,7)--(5,6)--(4,4)--cycle;
  \draw[>=stealth,->] (-1,2)--(0.5,3.5) node [midway, above left] {$u$};
\end{tikzpicture}

(b)
\end{center}
}
\parbox{0.25\textwidth}{
\begin{center}
\begin{tikzpicture}[scale=0.25]
\draw [black,fill=gray!40] (0,0)--(-3,3)--(5,11)--(7,12)--(10,11)--(9,9)--cycle ;
  \draw[dashed] (0,0)--(-3,3)--(-2,4)--(0,5)--(3,4)--(2,2)--cycle;
  \draw (0,0)--(-3,3)--(1,7)--(4,4)--cycle;
  \draw (-3,3)--(1,7) node[midway,above left] {$R$};
  \draw (1,7)--(4,4)--(8,8)--(5,11)--cycle;
  \draw (1,7)--(5,11) node[midway,above left] {$a_1u+R$};
  \draw (5,11)--(8,8)--(12,12)--(9,15)--cycle;
  \draw (5,11)--(9,15) node[midway,above left] {$2a_1u+R$};
  \draw [>=stealth,->] (1,-1)--(3,1) node[midway,below right]{$u$};
  \draw [>=stealth,->] (-1,-1)--(-3,1) node[midway,below left]{$u^\perp$};
\end{tikzpicture}

(c)
\end{center}
}

  \caption{The proof of proposition \ref{prop_Bollobas}. 
  (a) The mechanism for $d=1$; the $\bullet$ represent zeroes and the $\circ$ represent ones.  
  (b) The shape delimited by the solid line is the droplet of \cite{Bollobas_et_al2015}; 
  if it is infected in the bootstrap percolation dynamics, the infection 
  can grow to the shape delimited by the dashed line. (c) $R$ contains the 
  original droplet (dashed line), hence if $R$ is infected, the infection 
  can propagate to a bigger droplet (in gray) that contains $a_1u+R$ and is contained 
  in $R \cup (a_1u+R) \cup (2a_1u+R)$.}
  \label{fig_preuve_prop_Bollobas}
\end{figure}
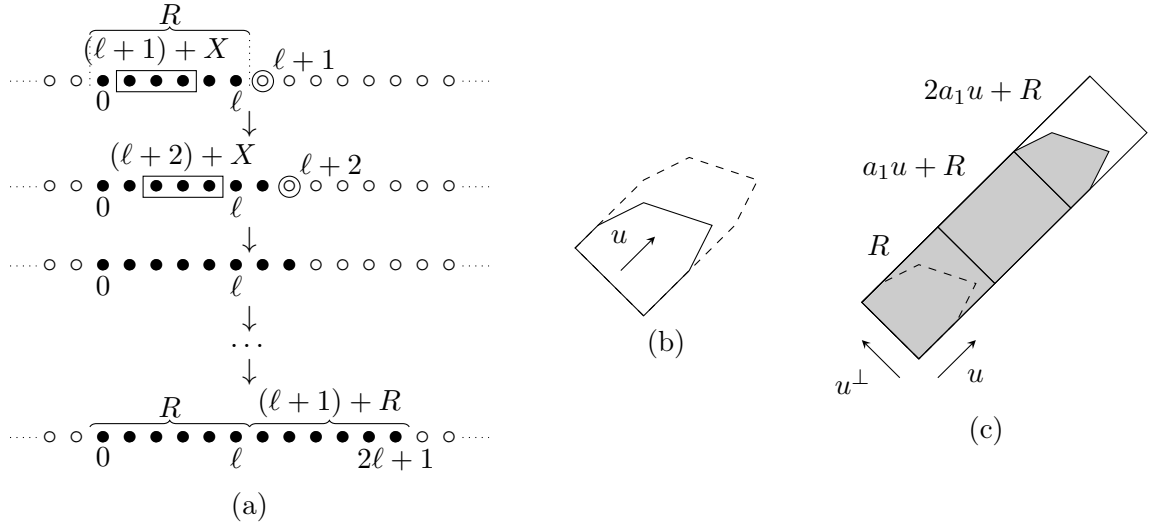
\end{proof}

\subsection{Definition of the auxiliary process}\label{subsec_def_aux_process}

Let $K > 0$, $q \in [0,1]$ and $t \geq K$. For any $y \in \mathds{Z}^d$ and 
$k \in \{0,\dots,\lfloor \frac{t}{K} \rfloor\}$, 
we will define an oriented percolation process $\zeta^{y,k}$ on $\mathds{Z}$, 
from time zero to time $n^{y,k} = \lfloor \frac{t}{K} \rfloor-k$ 
(see \cite{Durrett_1984} for an introduction to oriented percolation). 
For $n \in \{1,\dots, n^{y,k}\}$ and $r \in \mathds{Z}$ with $r+n$ even, the \emph{bonds}
$(r-1,n-1)\rightarrow (r,n)$ and $(r+1,n-1)\rightarrow (r,n)$ can be \emph{open} or \emph{closed}. 
We set $\zeta_0^{y,k}(r)=\mathds{1}_{\{r=0\}}$, and for any $n \in \{1,\dots, n^{y,k}\}$, 
$r \in \mathds{Z}$ with $r+n$ even, $\zeta_n^{y,k}(r)=1$ if and only if 
$\zeta_{n-1}^{y,k}(r-1)=1$ and the bond $(r-1,n-1)\rightarrow (r,n)$ is open 
or $\zeta_{n-1}^{y,k}(r+1)=1$ and the bond $(r+1,n-1)\rightarrow (r,n)$ is open. 
For any $n \in \{1,\dots, n^{y,k}\}$, $r \in \mathds{Z}$ with $r+n$ odd, we set $\zeta_n^{y,k}(r)=0$. 

The state of the bonds is defined as follows. For any $n \in \{1,\dots,n^{y,k}\}$, $r \in \mathds{Z}$ with $r+n$ even: 
\begin{itemize}
\item $(r-1,n-1)\rightarrow (r,n)$ is open if and only if 
\[
\left\{\forall x \in y+\frac{r-n}{2} a_1u+R, ]t-(k+n)K,t-(k+n-1)K]\cap\mathcal{P}^1_x=\emptyset\right\}, 
\] 
i.e. there is no 1-clock ring in $y+\frac{r-n}{2} a_1u+R$ during the time interval $]t-(k+n)K,t-(k+n-1)K]$; 
\item $(r+1,n-1)\rightarrow (r,n)$ is open if and only if 
\begin{gather*}
  \left\{\exists t-(k+n)K < t_1 < \cdots < t_m \leq t-(k+n-1)K, 
\forall i \in \{1,\dots,m\}, t_i \in \mathcal{P}^0_{y+\frac{r-n}{2} a_1u+x_i}\right\} \\ 
\cap \left\{\forall x \in y+\frac{r-n}{2} a_1u+(R \cup \{x_1,\dots,x_m\}), 
]t-(k+n)K,t-(k+n-1)K]\cap\mathcal{P}^1_x=\emptyset\right\},
\end{gather*} 
i.e. there are successive 0-clock rings in the equivalent of $x_1,\dots,x_m$ for $y+\frac{r-n}{2} a_1u+R$ 
during the time interval $]t-(k+n)K,t-(k+n-1)K]$, and no 1-clock ring at these sites or in $y+\frac{r-n}{2} a_1u+R$ 
in this time interval.
\end{itemize}

We notice that if all the sites of $y+\frac{r-n}{2} a_1u+R$ are at zero at time $t-(k+n)K$ 
and $(r-1,n-1)\rightarrow (r,n)$ is open, the sites of $y+\frac{r-n}{2} a_1u+R$ are still at zero 
at time $t-(k+n-1)K$. Moreover, by proposition \ref{prop_Bollobas}, if 
the sites of $y+\frac{r-n}{2} a_1u+R$ are at zero at time $t-(k+n)K$ and $(r+1,n-1)\rightarrow (r,n)$ is open, 
the sites of $a_1u+(y+\frac{r-n}{2} a_1u+R) = y+\frac{(r+1)-(n-1)}{2} a_1u+R$ are at zero at time $t-(k+n-1)K$.
This allows us to deduce (see figure \ref{fig_transfer_zeroes} for an illustration of the mechanism): 

\begin{proposition}\label{prop_transfer_zeroes2}
If there exists $r_0 \in \mathds{Z}$ such that 
$\zeta_{n^{y,k}}^{y,k}(r_0)=1$ and the sites of $y+\frac{r_0-n^{y,k}}{2}a_1u+R$ 
are at zero at time $t-\lfloor \frac{t}{K} \rfloor K$, then the sites of $y+R$ are at zero at time $t-kK$.
\end{proposition}

\begin{figure}
\begin{tikzpicture}
\draw (0,-3)--(0,3);
\draw[dashed] (0,-3)--(0,-3.5);
\draw[dashed] (0,3)--(0,3.5);
\draw (0,-3) node{$-$} node[right] {-3};
\draw (0,-2) node[right] {-2};
\draw (0,-1) node{$-$} node[right] {-1};
\draw (0,0) node[right] {0};
\draw (0,1) node{$-$} node[right] {1};
\draw (0,2) node[right] {2};
\draw (0,3) node{$-$} node[right] {3};
\draw (0,4) node{$r$} ;
\draw[->] (0,-3.8)--(-3,-3.8);
\draw (0,-3.8) node{$\shortmid$} node[below] {0};
\draw (-1,-3.8) node{$\shortmid$} node[below] {1};
\draw (-2,-3.8) node{$\shortmid$} node[below] {2};
\draw (-3,-3.8) node[below] {3};
\draw (-1.5,-4.5) node{$n$};
\draw[gray] (-1,-3)--(0,-2);
\draw[gray] (-3,-3)--(0,0);
\draw[gray] (-3,-1)--(0,2);
\draw[gray] (-3,1)--(-1,3);
\draw[gray] (-1,-3)--(-3,-1);
\draw[gray] (0,-2)--(-3,1);
\draw[gray] (0,0)--(-3,3);
\draw[gray] (0,2)--(-1,3);
\draw[gray,dashed] (-1,-3)--(-0.5,-3.5);
\draw[gray,dashed] (-1,-3)--(-1.5,-3.5);
\draw[gray,dashed] (-3,-3)--(-2.5,-3.5);
\draw[gray,dashed] (-1,3)--(-0.5,3.5);
\draw[gray,dashed] (-1,3)--(-1.5,3.5);
\draw[gray,dashed] (-3,3)--(-2.5,3.5);
\draw [ultra thick] (-1,-3)--(-2,-2);
\draw [ultra thick] (-3,-3)--(-2,-2);
\draw [ultra thick] (-3,-1)--(-2,-2);
\draw [ultra thick] (-1,-1)--(-2,0);
\draw [ultra thick] (-2,0)--(-3,1);
\draw [ultra thick] (-2,0)--(-1,1);
\draw [ultra thick] (-1,1)--(0,0);
\draw [ultra thick] (-2,2)--(-3,3);
\draw [ultra thick] (-1,3)--(0,2);
\draw[dashed, ultra thick] (-3,-3)--(-2.5,-3.5);
\draw (-3,1) node [left] {$r_0$};
\draw (-2,0)--(-3,1) node[midway, below left]{$b_1$};
\draw (-1,1)--(-2,0) node[midway, below right]{$b_2$};
\draw (0,0)--(-1,1) node[midway, above right]{$b_3$};
\draw[>=stealth,->] (0,0.2)--(-0.8,1) ;
\draw[>=stealth,->] (-1,0.8)--(-1.8,0) ;
\draw[>=stealth,->] (-2.2,0)--(-3,0.8) ;
\end{tikzpicture}
\hspace{\fill}
\begin{tikzpicture}
\draw (0,-3) node[left]{$y-3a_1u+R$};
\draw (0,-2) node[left]{$y-2a_1u+R$};
\draw (0,-1) node[left]{$y-a_1u+R$};
\draw (0,0) node[left]{$y+R$};
\draw (0,1) node[left]{$y+a_1u+R$};
\draw (0,2) node[left]{$y+2a_1u+R$};
\draw (0,3) node[left]{$y+3a_1u+R$};
\draw (0,-3.5)--(0,3.5);
\draw (1,-3.5)--(1,3.5);
\draw[dashed] (0,-4)--(0,-3.5);
\draw[dashed] (1,-4)--(1,-3.5);
\draw[dashed] (0,4)--(0,3.5);
\draw[dashed] (1,4)--(1,3.5);
\draw (0,-3.5)--(1,-3.5);
\draw (0,-2.5)--(1,-2.5);
\draw (0,-1.5)--(1,-1.5);
\draw (0,-0.5)--(1,-0.5);
\draw (0,0.5)--(1,0.5);
\draw (0,1.5)--(1,1.5);
\draw (0,2.5)--(1,2.5);
\draw (0,3.5)--(1,3.5);
\draw (0.5,-4.5) node{$t-(k+3)K$};
\draw (2.5,-3.5)--(2.5,3.5);
\draw (3.5,-3.5)--(3.5,3.5);
\draw[dashed] (2.5,-4)--(2.5,-3.5);
\draw[dashed] (3.5,-4)--(3.5,-3.5);
\draw[dashed] (2.5,4)--(2.5,3.5);
\draw[dashed] (3.5,4)--(3.5,3.5);
\draw (2.5,-3.5)--(3.5,-3.5);
\draw (2.5,-2.5)--(3.5,-2.5);
\draw (2.5,-1.5)--(3.5,-1.5);
\draw (2.5,-0.5)--(3.5,-0.5);
\draw (2.5,0.5)--(3.5,0.5);
\draw (2.5,1.5)--(3.5,1.5);
\draw (2.5,2.5)--(3.5,2.5);
\draw (2.5,3.5)--(3.5,3.5);
\draw (3,-4.5) node{$t-(k+2)K$};
\draw (5,-3.5)--(5,3.5);
\draw (6,-3.5)--(6,3.5);
\draw[dashed] (5,-4)--(5,-3.5);
\draw[dashed] (6,-4)--(6,-3.5);
\draw[dashed] (5,4)--(5,3.5);
\draw[dashed] (6,4)--(6,3.5);
\draw (5,-3.5)--(6,-3.5);
\draw (5,-2.5)--(6,-2.5);
\draw (5,-1.5)--(6,-1.5);
\draw (5,-0.5)--(6,-0.5);
\draw (5,0.5)--(6,0.5);
\draw (5,1.5)--(6,1.5);
\draw (5,2.5)--(6,2.5);
\draw (5,3.5)--(6,3.5);
\draw (5.5,-4.5) node{$t-(k+1)K$};
\draw (7.5,-3.5)--(7.5,3.5);
\draw (8.5,-3.5)--(8.5,3.5);
\draw[dashed] (7.5,-4)--(7.5,-3.5);
\draw[dashed] (8.5,-4)--(8.5,-3.5);
\draw[dashed] (7.5,4)--(7.5,3.5);
\draw[dashed] (8.5,4)--(8.5,3.5);
\draw (7.5,-3.5)--(8.5,-3.5);
\draw (7.5,-2.5)--(8.5,-2.5);
\draw (7.5,-1.5)--(8.5,-1.5);
\draw (7.5,-0.5)--(8.5,-0.5);
\draw (7.5,0.5)--(8.5,0.5);
\draw (7.5,1.5)--(8.5,1.5);
\draw (7.5,2.5)--(8.5,2.5);
\draw (7.5,3.5)--(8.5,3.5);
\draw (8,-4.5) node{$t-kK$};
\draw [black,fill=gray] (0,-0.5)--(1,-0.5)--(1,-1.5)--(0,-1.5)--cycle;
\draw[>=stealth,->] (1.2,-1)--(2.3,-1);
\draw [black,fill=gray] (2.5,-0.5)--(3.5,-0.5)--(3.5,-1.5)--(2.5,-1.5)--cycle;
\draw[>=stealth,->] (3.7,-1)--(4.8,0);
\draw [black,fill=gray] (5,0.5)--(6,0.5)--(6,-0.5)--(5,-0.5)--cycle;
\draw[>=stealth,->] (6.2,0)--(7.3,0);
\draw [black,fill=gray] (7.5,0.5)--(8.5,0.5)--(8.5,-0.5)--(7.5,-0.5)--cycle;
\end{tikzpicture}
\caption{An illustration of proposition \ref{prop_transfer_zeroes2} with $n^{y,k} = 3$ and $r_0=1$. 
The figure at the left represents the bonds of the oriented percolation 
process $\zeta^{y,k}$; the open bonds are the thick ones, and the path of open bonds allowing 
$\zeta_{n^{y,k}}^{y,k}(r_0)=1$ is outlined by arrows. 
The figure at the right represents the consequences on the KCM process; each vertical strip represents the state 
of $\bigcup_{i \in \mathds{Z}}(y+ia_1u+R)$ at a certain time. If at time $t-(k+3)K$ the 
rectangle $y+\frac{1-n^{y,k}}{2}a_1u+R = y-a_1u+R$ is at zero (in gray), since the bond 
$(0,2) \rightarrow (1,3)$ (bond $b_1$) is open, $y-a_1u+R$ is still at zero at time $t-(k+2)K$. Moreover, since 
$(1,1) \rightarrow (0,2)$ (bond $b_2$) is open and $y-a_1u+R$ is at zero at time $t-(k+2)K$, 
$a_1u+(y-a_1u+R) = y+R$ is at zero at time $t-(k+1)K$. Finally, since $(0,0) \rightarrow (1,1)$ (bond $b_3$) is open 
and $y+R$ is at zero at time $t-(k+1)K$, $y+R$ is still at zero at time $t-kK$.}
\label{fig_transfer_zeroes}
\end{figure}
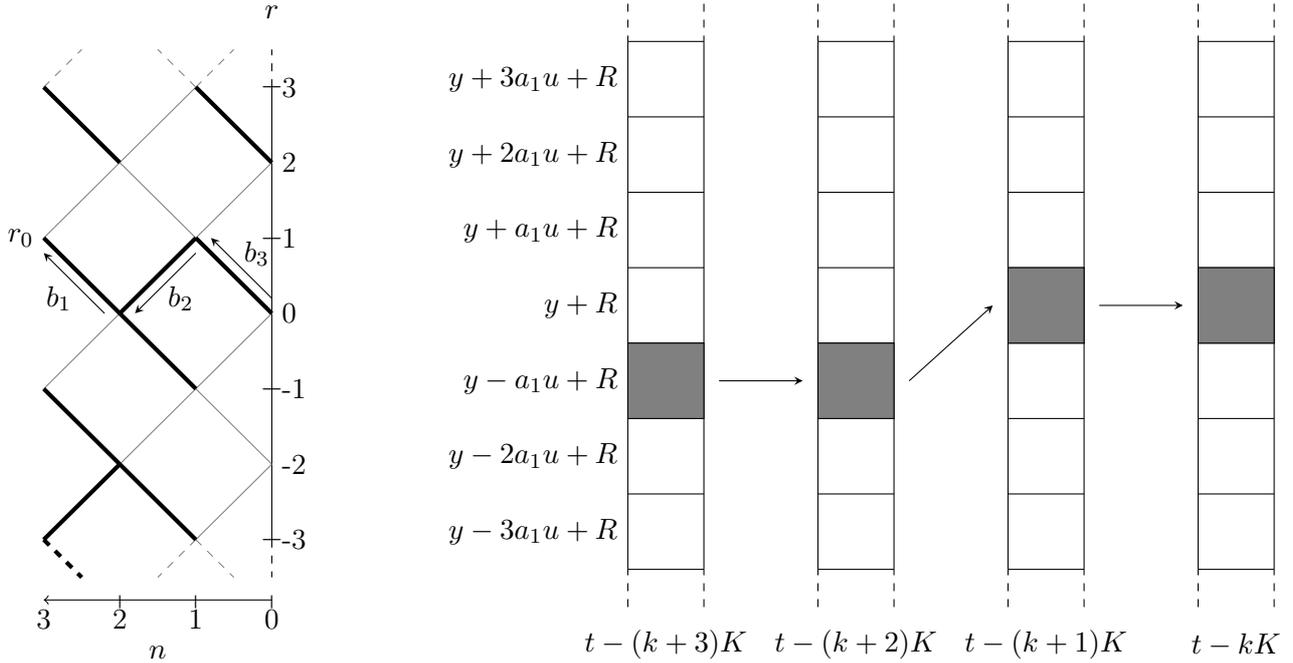

\subsection{Properties of the auxiliary process}\label{subsec_prop_aux_process}

In this subsection we state the two oriented percolation properties 
of $\zeta^{y,k}$, propositions \ref{prop_extinction_time} and \ref{prop_large_deviations}, that we will use in 
section \ref{sec_preuve_codings} to prove proposition \ref{prop_bound_single_coding}. 
In order to do that, we need a bound on the probability 
that a bond is closed; this will be lemma \ref{lemma_prob_bonds}. It is there that we need 
$q$ bigger than a $q_0 > 0$; this is necessary so that the probability 
that there is no 1-clock ring at the sites we consider 
is large. For any $K > 0$, we set $q_K = 1+\frac{1}{3K|R|}\ln(1-e^{-K})$. We can then state 

\begin{lemma}\label{lemma_prob_bonds}
There exists a constant $K_p=K_p(\mathcal{U}) > 0$ such that for $K \geq K_p$, 
$q \in [q_K,1]$, $t \geq K$, $y \in \mathds{Z}^d$ and $k \in \{0,\dots,\lfloor \frac{t}{K} \rfloor\}$, 
the probability that any given bond is closed for the process $\zeta^{y,k}$
is smaller than $e^{-\frac{K}{4}}$.
\end{lemma}

\begin{proof}
Let $K > 0$, $q \in [q_K,1]$, $t \geq K$, $y \in \mathds{Z}^d$ 
  and $k \in \{0,\dots,\lfloor \frac{t}{K} \rfloor\}$. 
Let $n \in \{1,\dots,n^{y,k}\}$, $r \in \mathds{Z}$ with $r+n$ even. 
We notice that if the bond $(r-1,n-1)\rightarrow (r,n)$ is closed, the bond $(r+1,n-1)\rightarrow (r,n)$ 
is also closed, hence it is enough to bound the probability that $(r+1,n-1)\rightarrow (r,n)$ is closed. 
Denoting $E_1 = \{\forall x \in y+\frac{r-n}{2} a_1u+R \cup \{x_1,\dots,x_m\}, 
]t-(k+n)K,t-(k+n-1)K]\cap\mathcal{P}^1_x=\emptyset\}$ and 
$E_2 = \{\exists t-(k+n)K < t_1 < \cdots < t_m \leq t-(k+n-1)K, 
\forall i \in \{1,\dots,m\}, t_i \in \mathcal{P}^0_{y+\frac{r-n}{2} a_1u+x_i}\}$, we need to bound the 
probabilities of $E_1^c$ and $E_2^c$. We begin with $E_1^c$. The events $]t-(k+n)K,t-(k+n-1)K]\cap\mathcal{P}^1_x=\emptyset$ are 
independent and have probability $e^{-(1-q)K}$ each; moreover, $x_1,\dots,x_m$ belong to 
$(a_1u+R) \cup (2a_1u+R)$, so $\left|R\cup\{x_1,\dots,x_m\}\right| \leq 3|R|$; 
we deduce the probability of $E_1$ is  
\begin{gather*}
e^{-\left|R\cup\{x_1,\dots,x_m\}\right|(1-q)K} \geq e^{-3|R|(1-q)K} \geq e^{-3|R|(1-q_K)K} \\
  \geq e^{-3|R|\left(1-\left(1+\frac{1}{3K|R|}\ln(1-e^{-K})\right)\right)K} = e^{\ln(1-e^{-K})} = 1-e^{-K},
\end{gather*}
thus the probability of $E_1^c$ is at most $e^{-K}$.
Moreover, the probability of $E_2^c$ is the probability 
that a Poisson point process of parameter $q$ has strictly less than 
$m$ elements in an interval of length $K$, hence it is $\sum_{i=0}^{m-1} e^{-qK}\frac{(qK)^i}{i!}$. 
When $K$ is large enough, $q \in [1/2,1]$, hence this probability is smaller than 
$e^{-\frac{1}{2}K} \sum_{i=0}^{m-1} \frac{K^i}{i!}$, which is smaller than $e^{-\frac{K}{3}}$ 
when $K$ is large enough depending on $m$, hence on $\mathcal{U}$.
Consequently, when $K$ is large enough depending on $\mathcal{U}$, the probability that 
$(r+1,n-1)\rightarrow (r,n)$ is closed is smaller than $e^{-K} + e^{-\frac{K}{3}}$, 
which is smaller than $e^{-\frac{K}{4}}$.
\end{proof}

Thanks to lemma \ref{lemma_prob_bonds}, it is possible to prove two oriented percolation properties 
of $\zeta^{y,k}$. Firstly, for any $K > 0$, $q \in [q_K,1]$, $t \geq K$, $y \in \mathds{Z}^d$ 
and $k \in \{0,\dots,\lfloor \frac{t}{K} \rfloor\}$, 
we define $\tau^{y,k}=\inf\{n \in \{0,\dots,n^{y,k}\} 
\,|\, \forall r \in \mathds{Z}, \zeta^{y,k}_n(r)=0\}$ 
the time of death of the process $\zeta^{y,k}$
(if the set is empty, $\tau^{y,k}$ is infinite). Since $\zeta_0^{y,k}(r)=\mathds{1}_{\{r=0\}}$, which 
is not identically zero, $\tau^{y,k} \geq 1$. Then we have 

\begin{proposition}\label{prop_extinction_time}
For any $q' \in [0,1]$, there exists a constant $K_c=K_c(\mathcal{U})>0$ such that for any $K \geq K_c$, 
$q \in [q_K,1]$, $t \geq K$, $y \in \mathds{Z}^d$, 
$k \in \{0,\dots,\lfloor \frac{t}{K} \rfloor\}$, 
$n \in \{0,\dots,n^{y,k}\}$, 
$\mathds{P}_{q',q}(n \leq \tau^{y,k} < + \infty) \leq 2. 3^{2n} e^{-\frac{Kn}{24}}$.
\end{proposition}

\begin{proof}[Sketch of proof.]
The proposition can be proven by a classical contour method like the one presented in section 10 of 
\cite{Durrett_1984}. The idea is that if $n \leq \tau^{y,k} < + \infty$ we can draw a 
``contour of closed bonds'' around the connected component of ones in $\zeta^{y,k}$, 
and this contour will have length $\Omega(n)$. Furthermore, it can be seen that bonds separated by at least 5 bonds  
from each other are independent, because they depend on clock rings in disjoint space-time intervals. 
Therefore if we keep one bond out of 6, we extract $\Omega(n)$ independent closed bonds from the contour, each of them 
having probability $e^{-\frac{K}{4}}$ from lemma \ref{lemma_prob_bonds} when $K \geq K_p$, hence the bound.
\end{proof}

$\zeta^{y,k}$ also satisfies a second property. For any $K > 0$, $q \in [q_K,1]$, $t \geq K$, 
$y \in \mathds{Z}^d$ and $k \in \{0,\dots,\lfloor \frac{t}{K} \rfloor\}$,
we denote $\mathcal{X}^{y,k}=\{r \in \{-\lfloor \frac{n^{y,k}}{2} \rfloor,\dots,
\lfloor \frac{n^{y,k}}{2} \rfloor\} \,|\, \zeta^{y,k}_{n^{y,k}}(r) = 1\}$. 
Then we have 

\begin{proposition}\label{prop_large_deviations}
For any $q' \in [0,1]$, $\alpha \in ]0,1[$, there exists a constant 
$K_g(\alpha) = K_g(\mathcal{U},\alpha) > 0$ such that for 
any $K \geq K_g(\alpha)$, there exist constants $c_g > 0$ and $C_g = C_g(\mathcal{U},K,\alpha) > 0$ 
such that for any $q \in [q_K,1]$, $t \geq K$, 
$y \in \mathds{Z}^d$ and $k \in \{0,\dots,\lfloor \frac{t}{K} \rfloor\}$, 
$\mathds{P}_{q',q}\left(\tau^{y,k}=+\infty, |\mathcal{X}^{y,k}| \leq \frac{\alpha}{2} n^{y,k}\right) 
\leq C_ge^{-c_g n^{y,k}}$.
\end{proposition}

\begin{proof}[Sketch of proof.]
This proposition comes from classical results in oriented percolation. Firstly, 
if the process survives until time $n^{y,k}$, it has a big ``range'', which means that 
if we define $r^{y,k}=\sup\{r \in \mathds{Z} \,|\, \zeta^{y,k}_{n^{y,k}}(r)=1\}$ and 
$\ell^{y,k}=\inf\{r \in \mathds{Z} \,|\, \zeta^{y,k}_{n^{y,k}}(r)=1\}$, $r^{y,k}$ 
and $|\ell^{y,k}|$ are so large $\{-\lfloor \frac{n^{y,k}}{2} \rfloor,\dots,
\lfloor \frac{n^{y,k}}{2} \rfloor\} \subset \{\ell^{y,k},\dots,r^{y,k}\}$; 
this can be proven with the contour argument in section 11 of \cite{Durrett_1984}. 
  Moreover, the argument that proves (1) in \cite{Durrett_1984} also proves that in 
  $\{\ell^{y,k},\dots,r^{y,k}\}$, $\zeta^{y,k}_{n^{y,k}}$ coincides with the oriented percolation 
  process that has the same bonds, but which starts with all sites at 1 instead of just the origin. 
  Finally, the end of section 5 of \cite{Durrett_et_al1988} contains a contour argument for the latter 
  process which allows to prove that it has a lot of ones; we can use this argument with the same 
  adaptations we used for the contours of proposition~\ref{prop_extinction_time}.
\end{proof}

\section{Proof of proposition \ref{prop_bound_single_coding}}\label{sec_preuve_codings}

In this section we use the auxiliary process defined in section \ref{sec_aux_proc} 
to give a proof of proposition \ref{prop_bound_single_coding}. In order to do that, we need some definitions. 
For any $q' \in ]0,1]$, $K \geq 2$, $q \in [q_K,1]$, $x \in \mathds{Z}^d$, $t \geq K$ and 
$\gamma = (y_k)_{k \in \{0,\dots,\lfloor\frac{t}{K^2}\rfloor\}} \in C_K^N(x,t)$, 
we define $k(\gamma)=\inf\{k \in \{0,\dots,\lfloor\frac{t}{K^2}\rfloor\} 
\,|\, \tau^{y_k,k}=+\infty\}$ if such a $k$ exists; 
in this case we also denote $y(\gamma)=y_{k(\gamma)}$ (in the following, when we write $k(\gamma)$ or 
$y(\gamma)$ without more precision, we always assume that they exist). For any 
$r \in \mathcal{X}^{y(\gamma),k(\gamma)}$ we define the events 
\[
W^{\gamma,\eta}(r) = 
\left\{(\eta_{t-\lfloor\frac{t}{K}\rfloor K})_{y(\gamma)+\frac{r-n^{y(\gamma),k(\gamma)}}{2}a_1u+R} = 0\right\}, 
W^{\gamma,\tilde{\eta}}(r) = 
\left\{(\tilde{\eta}_{t-\lfloor\frac{t}{K}\rfloor K})_{y(\gamma)+\frac{r-n^{y(\gamma),k(\gamma)}}{2}a_1u+R} = 0\right\}.
\]
By proposition \ref{prop_transfer_zeroes2}, if 
$\{\exists r \in \mathcal{X}^{y(\gamma),k(\gamma)}, W^{\gamma,\eta}(r)\} \cap 
\{\exists r \in \mathcal{X}^{y(\gamma),k(\gamma)},W^{\gamma,\tilde{\eta}}(r)\}$,  
then the sites of $y(\gamma) + R$ are at zero at time $t - k(\gamma) K$ in both processes 
$(\eta_t)_{t \in [0,+\infty[}$ and $(\tilde{\eta}_t)_{t \in [0,+\infty[}$, 
in particular $y(\gamma)$ is at zero at time $t - k(\gamma) K$ in both processes, 
therefore $G(\gamma)$ is satisfied. Consequently, 
\begin{align*}
  \mathds{P}_{q',q}(G(\gamma)^c) 
  \leq & \mathds{P}_{q',q}\left(k(\gamma)\text{ does not exist}\right)
  +\mathds{P}_{q',q}\left(k(\gamma)\text{ exists}, |\mathcal{X}^{y(\gamma),k(\gamma)}| \leq \frac{t}{6K}\right)\\
  &+\mathds{P}_{q',q}\left(\left\{|\mathcal{X}^{y(\gamma),k(\gamma)}| > \frac{t}{6K}\right\} \cap 
  \left\{\forall r \in \mathcal{X}^{y(\gamma),k(\gamma)}, 
  W^{\gamma,\eta}(r)^c\right\} \right) \\
  &+\mathds{P}_{q',q}\left(\left\{|\mathcal{X}^{y(\gamma),k(\gamma)}| > \frac{t}{6K}\right\} \cap 
  \left\{\forall r \in \mathcal{X}^{y(\gamma),k(\gamma)}, 
  W^{\gamma,\tilde{\eta}}(r)^c\right\} \right).
\end{align*}
Therefore we only have to prove the following lemmas \ref{lemma_percolation_structure}, 
\ref{lemma_percolation_use} and \ref{lemma_wonderful_rectangles} to prove proposition 
\ref{prop_bound_single_coding}, thus ending the proof of theorem \ref{thm_convergence}: 

\begin{lemma}\label{lemma_percolation_structure}
For any $q' \in ]0,1]$, there exists a constant $K_1=K_1(\mathcal{U}) \geq 2$ such that for any $K \geq K_1$, 
$q \in [q_K,1]$, there exist constants $\breve{c}_1 > 0$ 
and $\breve{C}_1=\breve{C}_1(K) > 0$ such that for any $x \in \mathds{Z}^d$, $t \geq K$, $\gamma \in C_K^N(x,t)$, we have 
$\mathds{P}_{q',q}(k(\gamma)$ does not exist$)\leq \breve{C}_1e^{-\breve{c}_1\frac{t}{K}}$.
\end{lemma}

\begin{lemma}\label{lemma_percolation_use}
For any $q' \in ]0,1]$, there exists a constant $K_2=K_2(\mathcal{U}) \geq 2$ such that 
for any $K \geq K_2$, $q \in [q_K,1]$,  there exist constants 
$\breve{c}_2 > 0$ and $\breve{C}_2=\breve{C}_2(\mathcal{U},K) > 0$ such that for any 
$x \in \mathds{Z}^d$, $t \geq K$, $\gamma \in C_K^N(x,t)$, 
$\mathds{P}_{q',q}(k(\gamma)$ exists, $|\mathcal{X}^{y(\gamma),k(\gamma)}| 
\leq \frac{t}{6K})\leq \breve{C}_2e^{-\breve{c}_2\frac{t}{K}}$.
\end{lemma}

\begin{lemma}\label{lemma_wonderful_rectangles}
For any $q' \in ]0,1]$, $K \geq 2$, $q \in [q_K,1]$,  
there exists a constant $\breve{c}_3=\breve{c}_3(\mathcal{U},q') > 0$ 
such that for any $x \in \mathds{Z}^d$, $t \geq K$, $\gamma \in C_K^N(x,t)$, we get 
$\mathds{P}_{q',q}(\{|\mathcal{X}^{y(\gamma),k(\gamma)}| > \frac{t}{6K}\} \cap 
  \{\forall r \in \mathcal{X}^{y(\gamma),k(\gamma)}, 
  W^{\gamma,\eta}(r)^c\})\leq e^{-\breve{c}_3\frac{t}{K}}$ and 
  $\mathds{P}_{q',q}(\{|\mathcal{X}^{y(\gamma),k(\gamma)}| > \frac{t}{6K}\} \cap 
  \{\forall r \in \mathcal{X}^{y(\gamma),k(\gamma)}, 
  W^{\gamma,\tilde{\eta}}(r)^c\})\leq e^{-\breve{c}_3\frac{t}{K}}$.
\end{lemma}

\begin{proof}[Proof of lemma \ref{lemma_percolation_structure}.]
We set $K_1 = \max(K_c,48(\ln36+1))$, which depends only on $\mathcal{U}$.
Let $q' \in ]0,1]$, $K \geq K_1$, $q \in [q_K,1]$, 
$x \in \mathds{Z}^d$, $t \geq K$ and $\gamma=(y_k)_{k \in \{0,\dots,\lfloor\frac{t}{K^2}\rfloor\}} 
\in C_K^N(x,t)$. If $k(\gamma)$ does not exist, $\tau^{y_k,k}$ is finite 
for $k \in \{0,\dots,\lfloor\frac{t}{K^2}\rfloor\}$, 
therefore if we call $k_1=0$  and $k_i=\sum_{j=1}^{i-1}\tau^{y_{k_j},k_j}$ for $i \geq 2$, 
$\tau^{y_{k_i},k_i}$ is finite as long as 
$k_i \leq \lfloor\frac{t}{K^2}\rfloor$. We will use 
proposition \ref{prop_extinction_time} to bound the probability that this happens. 
We call $L = \max\{i \geq 1 \,|\,k_i \leq \lfloor\frac{t}{K^2}\rfloor\}$; 
we then have $\sum_{i=1}^{L}\tau^{y_{k_i},k_i} > \lfloor\frac{t}{K^2}\rfloor$, hence if $n_L$ is the integer satisfying 
$n_L = \lfloor\frac{t}{K^2}\rfloor - \sum_{i=1}^{L-1}\tau^{y_{k_i},k_i}$, 
we have $n_L \leq \tau^{y_{k_L},k_L} < +\infty$. Furthermore, if $n_1,\dots,n_{L-1}$ are integers satisfiying 
$n_i = \tau^{y_{k_i},k_i}$ for $i \in \{1,\dots,L-1\}$, we get 
$n_1+\cdots+n_L=\lfloor\frac{t}{K^2}\rfloor$, $k_i = \sum_{j=1}^{i-1}n_{j}$ 
for all $i \in \{1,\dots,L\}$ (we denote $\sum_{j=1}^{i-1}n_{j}=N_i$). In addition, 
since $\tau^{y_{k},k} \geq 1$ for any $k \in \{0,\dots,\lfloor\frac{t}{K^2}\rfloor\}$, 
$L \leq \lfloor\frac{t}{K^2}\rfloor+1$. We deduce 
\[
\mathds{P}_{q',q}(k(\gamma)\text{ does not exist}) 
\]
\[
\leq \sum_{M \leq \left\lfloor\frac{t}{K^2}\right\rfloor+1,n_1+\cdots+n_M=\left\lfloor\frac{t}{K^2}\right\rfloor}
\mathds{P}_{q',q}(L=M,\forall 1 \leq i \leq M-1,\tau^{y_{N_i},N_i} =n_i, n_M \leq \tau^{y_{N_M},N_M} <+\infty).
\]
Moreover, the events $\{\tau^{y_{k_{N_i}},N_i}=n_i\}$, 
$i \in \{1,\dots,M-1\}$ and $\{n_M \leq \tau^{y_{k_{N_M}},N_M} <+\infty\}$ 
depend only on clock rings in the time intervals $]t-(N_i+n_i)K,t-N_iK] = ]t-N_{i+1}K,t-N_iK]$, 
$i \in \{1,\dots,M-1\}$ and $]t-(N_M+n_M)K,t-N_MK]$, 
which are disjoint, thus the events are independent, hence
\[
\mathds{P}_{q',q}(L=M,\forall 1 \leq i \leq M-1,\tau^{y_{N_i},N_i} =n_i, n_M \leq \tau^{y_{N_M},N_M} <+\infty)
\]
\[
\leq \left(\prod_{i=1}^{M-1}\mathds{P}_{q',q}\left(\tau^{y_{N_i},N_i}=n_i\right)\right)
\mathds{P}_{q',q}\left(n_M \leq \tau^{y_{N_M},N_M} <+\infty\right) 
\leq \prod_{i=1}^{M}\mathds{P}_{q',q}\left(n_i \leq \tau^{y_{N_i},N_i} <+\infty\right)
\]
\[
\leq \prod_{i=1}^{M} 2 . 3^{2n_i}e^{-\frac{Kn_i}{24}}
= 2^M 3^{2 \sum_{i=1}^M n_i}e^{-\frac{K}{24}\sum_{i=1}^M n_i}
= 2^M 3^{2 \left\lfloor\frac{t}{K^2}\right\rfloor}e^{-\frac{K}{24}\left\lfloor\frac{t}{K^2}\right\rfloor}
\]
by proposition \ref{prop_extinction_time} and since $n_1+\cdots+n_M=\left\lfloor\frac{t}{K^2}\right\rfloor$.
Consequently, 
\[
\mathds{P}_{q',q}(k(\gamma)\text{ does not exist}) \leq 
\sum_{M \leq \left\lfloor\frac{t}{K^2}\right\rfloor+1,n_1+\cdots+n_M=\left\lfloor\frac{t}{K^2}\right\rfloor} 
2^M 3^{2 \left\lfloor\frac{t}{K^2}\right\rfloor}e^{-\frac{K}{24}\left\lfloor\frac{t}{K^2}\right\rfloor}.
\]
In addition, lemma \ref{lemma_binomial_coeffs} yields that for any 
$M \in \{1,\dots,\lfloor\frac{t}{K^2}\rfloor+1\}$, we have 
$|\{(n_1,\dots,n_M)\in\mathds{N}^M \,|\, n_1+\cdots+n_M=\lfloor\frac{t}{K^2}\rfloor\}| = 
\binom{M+\lfloor\frac{t}{K^2}\rfloor-1}{M-1} = 
\binom{M+\lfloor\frac{t}{K^2}\rfloor-1}{\lfloor\frac{t}{K^2}\rfloor}$, and 
by the Stirling formula there exists a constant $\lambda > 0$ such that 
\[
\binom{M+\left\lfloor\frac{t}{K^2}\right\rfloor-1}{\left\lfloor\frac{t}{K^2}\right\rfloor} 
\leq \frac{\left(M+\left\lfloor\frac{t}{K^2}\right\rfloor-1\right)^{\!\left\lfloor\frac{t}{K^2}\right\rfloor}}
{\left\lfloor\frac{t}{K^2}\right\rfloor!}
\leq \lambda\!\left(\!\frac{e\left(M+\left\lfloor\frac{t}{K^2}\right\rfloor-1\right)}
{\left\lfloor\frac{t}{K^2}\right\rfloor}\!\right)^{\!\left\lfloor\frac{t}{K^2}\right\rfloor} 
\leq \lambda\!\left(\!\frac{e\left(\lfloor\frac{t}{K^2}\rfloor+\left\lfloor\frac{t}{K^2}\right\rfloor\right)}
{\left\lfloor\frac{t}{K^2}\right\rfloor}\!\right)^{\!\left\lfloor\frac{t}{K^2}\right\rfloor} 
\]
since $M \leq \lfloor\frac{t}{K^2}\rfloor+1$. We deduce 
$|\{(n_1,\dots,n_M)\in\mathds{N}^M \,|\, n_1+\cdots+n_M=\lfloor\frac{t}{K^2}\rfloor\}| \leq 
\lambda(2e)^{\lfloor\frac{t}{K^2}\rfloor}$.
Therefore 
\[
\mathds{P}_{q',q}(k(\gamma)\text{ does not exist}) \leq 
\sum_{M=1}^{\left\lfloor\frac{t}{K^2}\right\rfloor+1} \lambda(2e)^{\left\lfloor\frac{t}{K^2}\right\rfloor}
2^M 3^{2 \left\lfloor\frac{t}{K^2}\right\rfloor}e^{-\frac{K}{24}\left\lfloor\frac{t}{K^2}\right\rfloor} 
\]
\[
\leq \lambda(2e)^{\left\lfloor\frac{t}{K^2}\right\rfloor}
2^{\left\lfloor\frac{t}{K^2}\right\rfloor+2} 3^{2 \left\lfloor\frac{t}{K^2}\right\rfloor}
e^{-\frac{K}{24}\left\lfloor\frac{t}{K^2}\right\rfloor} 
= 4\lambda\left(36ee^{-\frac{K}{24}}\right)^{\left\lfloor\frac{t}{K^2}\right\rfloor}.
\]
In addition, since $K \geq 48(\ln36+1)$, $36ee^{-\frac{K}{48}} \leq 36ee^{-\ln36-1} = 1$, so 
$36ee^{-\frac{K}{24}} \leq e^{-\frac{K}{48}}$, hence 
\[
\mathds{P}_{q',q}(k(\gamma)\text{ does not exist}) 
\leq 4\lambda e^{-\frac{K}{48}\left\lfloor\frac{t}{K^2}\right\rfloor} 
\leq 4\lambda e^{-\frac{K}{48}\left(\frac{t}{K^2}-1\right)}
= 4\lambda e^{\frac{K}{48}}e^{-\frac{t}{48 K}},
\]
which is the lemma.
\end{proof}

\begin{proof}[Proof of lemma \ref{lemma_percolation_use}.] 
This proof is an application of proposition \ref{prop_large_deviations}.
We set $K_2 = \max (4,K_g(1/2))$, which depends only on 
$\mathcal{U}$. Let $q' \in ]0,1]$, $K \geq K_2$, $q \in [q_K,1]$ and $x \in \mathds{Z}^d$. It is enough 
to prove the lemma for $t \geq \max(K,\frac{3K^2}{K-3})$; indeed, if the lemma holds for 
$t \geq \max(K,\frac{3K^2}{K-3})$, one has only to enlarge $\breve{C}_2$ to prove it for $t \geq K$. 
Therefore we set $t \geq \max(K,\frac{3K^2}{K-3})$ and 
$\gamma = (y_k)_{k \in \{0,\dots,\lfloor\frac{t}{K^2}\rfloor\}} \in C_K^N(x,t)$.
If $k(\gamma)$ exists but $|\mathcal{X}^{y(\gamma),k(\gamma)}| \leq \frac{t}{6K}$, we have 
$\tau^{y(\gamma),k(\gamma)} = +\infty$ and $|\mathcal{X}^{y(\gamma),k(\gamma)}| \leq \frac{t}{6K}$, 
hence 
\[
\mathds{P}_{q',q}\left(k(\gamma)\text{ exists}, |\mathcal{X}^{y(\gamma),k(\gamma)}| 
\leq \frac{t}{6K}\right) 
\leq \sum_{k=0}^{\left\lfloor\frac{t}{K^2}\right\rfloor}
\mathds{P}_{q',q}\left(\tau^{y_k,k} = +\infty, |\mathcal{X}^{y_k,k}| \leq \frac{t}{6K}\right). 
\]
We are going to bound the term on the right. For any $k \in \{0,\dots,\lfloor\frac{t}{K^2}\rfloor\}$, 
we have $n^{y_k,k} = \lfloor\frac{t}{K}\rfloor-k \geq 
\lfloor\frac{t}{K}\rfloor-\lfloor\frac{t}{K^2}\rfloor \geq \frac{t}{K}-1-\frac{t}{K^2}$, and since 
$t \geq \frac{3K^2}{K-3}$, $(K-3)t \geq 3K^2$ thus $\frac{1}{3}\frac{t}{K}-\frac{t}{K^2} \geq 1$, so 
$n^{y_k,k} \geq \frac{2}{3}\frac{t}{K}$, hence if we choose $\alpha=\frac{1}{2}$ we have 
$\frac{\alpha}{2}n^{y_k,k} \geq \frac{t}{6K}$. Therefore by proposition \ref{prop_large_deviations}, 
\[
\mathds{P}_{q',q}\left(\tau^{y_k,k} = +\infty, |\mathcal{X}^{y_k,k}| \leq \frac{t}{6K}\right)
\leq C_g e^{-c_g n^{y_k,k}} \leq C_g e^{-c_g \frac{2}{3}\frac{t}{K}}
\]
since $n^{y_k,k} \geq \frac{2}{3}\frac{t}{K}$. Consequently 
\[
\mathds{P}_{q',q}\left(k(\gamma)\text{ exists}, |\mathcal{X}^{y(\gamma),k(\gamma)}| \leq \frac{t}{6K}\right) 
\leq \left(\left\lfloor\frac{t}{K^2}\right\rfloor+1\right)C_g e^{-\frac{2c_g}{3}\frac{t}{K}} 
\leq \left(\frac{t}{K}+1\right)C_g e^{-\frac{2c_g}{3}\frac{t}{K}}, 
\]
which yields lemma \ref{lemma_percolation_use}.
\end{proof}

\begin{proof}[Proof of lemma \ref{lemma_wonderful_rectangles}.]
Let $q' \in ]0,1]$, $K \geq 2$, $q \in [q_K,1]$, 
$x \in \mathds{Z}^d$, $t \geq K$ and $\gamma \in C_K^N(x,t)$. 
The argument is elementary: we notice that there is a positive probability that a rectangle is full of zeroes 
in the initial configurations of the two processes since they have laws $\nu_{q'}$ and $\nu_q$, 
as well as a positive probability that there is no 1-clock ring in the rectangle in the time interval 
$[0,t-K\lfloor\frac{t}{K}\rfloor]$. Therefore there is a positive probability 
that a rectangle is full of zeroes in both processes at time $t-K\lfloor\frac{t}{K}\rfloor$, so if there are 
$\frac{t}{6K}$ elements in $\mathcal{X}^{y(\gamma),k(\gamma)}$, the probability that none of the corresponding 
rectangles is full of zeroes in both processes at time $t-K\lfloor\frac{t}{K}\rfloor$ is of order 
$e^{-\breve{c}_3\frac{t}{K}}$. 

  We notice that $\mathcal{X}^{y(\gamma),k(\gamma)}$ depends only on clock rings 
  in the time interval $]t-K\lfloor\frac{t}{K}\rfloor,t]$, hence if $\mathcal{F}$ 
  is the $\sigma$-algebra generated by the clock rings in $]t-K\lfloor\frac{t}{K}\rfloor,t]$, 
  for $\hat{\eta}=\eta$ or $\tilde{\eta}$, we have 
  \begin{equation}\label{eq_wonderful_rectangles}
  \begin{split}
  \mathds{P}_{q',q}\left(\left\{|\mathcal{X}^{y(\gamma),k(\gamma)}| > \frac{t}{6K}\right\} \cap
  \{\forall r \in \mathcal{X}^{y(\gamma),k(\gamma)}, 
  W^{\gamma,\hat{\eta}}(r)^c\}\right) \\
  = \mathds{E}_{q',q}\left(\mathds{1}_{\{|\mathcal{X}^{y(\gamma),k(\gamma)}| > \frac{t}{6K}\}}
  \mathds{P}_{q',q}(\forall r \in \mathcal{X}^{y(\gamma),k(\gamma)}, 
  W^{\gamma,\hat{\eta}}(r)^c | \mathcal{F})\right).
  \end{split}
\end{equation}

Moreover, 
  \[
  \mathds{P}_{q',q}(\forall r \in \mathcal{X}^{y(\gamma),k(\gamma)}, 
  W^{\gamma,\hat{\eta}}(r)^c | \mathcal{F}) 
  \]
\[
= \mathds{P}_{q',q} \left(\left.\forall r \in \mathcal{X}^{y(\gamma),k(\gamma)}, 
\exists x' \in y(\gamma)+\frac{r-n^{y(\gamma),k(\gamma)}}{2}a_1u+R,
\hat{\eta}_{t-\lfloor\frac{t}{K}\rfloor K}(x') \neq 0 \right| \mathcal{F}\right) \leq 
\]
\[
\mathds{P}_{q',q}\left(\!\forall r \in \mathcal{X}^{y(\gamma),k(\gamma)}, 
\exists x'\in y(\gamma)+\frac{r-n^{y(\gamma),k(\gamma)}}{2}a_1u+R,  
\left. \hat{\eta}_0(x') \neq 0 \text{ or } \mathcal{P}_{x'}^1 \cap  
\left[0,t-\left\lfloor\frac{t}{K}\right\rfloor K\right] \neq \emptyset 
  \right| \mathcal{F}\!\right)
\]
\[
  = \prod_{r \in \mathcal{X}^{y(\gamma),k(\gamma)}} \mathds{P}_{q',q} \left(
  \exists x' \in y(\gamma)+\frac{r-n^{y(\gamma),k(\gamma)}}{2}a_1u+R,
  \hat{\eta}_0(x') \neq 0 \text{ or } \mathcal{P}_{x'}^1 \cap 
\left[0,t-\left\lfloor\frac{t}{K}\right\rfloor K\right] \neq \emptyset \right)
\]
since the events $\{\exists x' \in y(\gamma)+\frac{r-n^{y(\gamma),k(\gamma)}}{2}a_1u+R,\hat{\eta}_0(x') \neq 0$ or 
$\mathcal{P}_{x'}^1 \cap [0,t-\lfloor\frac{t}{K}\rfloor K] \neq 
\emptyset \}$ depend only on the state of $\hat{\eta}_0$ and on the 
clock rings of the time interval $[0,t-K\lfloor\frac{t}{K}\rfloor]$ at the sites of 
$y(\gamma)+\frac{r-n^{y(\gamma),k(\gamma)}}{2}a_1u+R$, so they are mutually independent and independent of 
$\mathcal{F}$. Therefore the invariance by translation yields 
\[
  \mathds{P}_{q',q}(\forall r \in \mathcal{X}^{y(\gamma),k(\gamma)}, 
  W^{\gamma,\hat{\eta}}(r)^c | \mathcal{F}) 
\leq \mathds{P}_{q',q}\left(\exists x' \in R,\hat{\eta}_0(x') \neq 0\text{ or }
\mathcal{P}_{x'}^1 \cap \left[0,t-\!\left\lfloor\frac{t}{K}\right\rfloor\! K\right] 
\neq \emptyset\right)^{|\mathcal{X}^{y(\gamma),k(\gamma)}|}
\]
\[
= \left(1-\mathds{P}_{q',q}\left(\forall x' \in R,\hat{\eta}_0(x') = 0,
\mathcal{P}_{x'}^1 \cap \left[0,t-\left\lfloor\frac{t}{K}\right\rfloor K\right] 
= \emptyset\right)\right)^{|\mathcal{X}^{y(\gamma),k(\gamma)}|}
\]
\[
  = \left(1-\left(\mathds{P}_{q',q}\left(\hat{\eta}_0(0) = 0\right)\mathds{P}_{q',q}\left(
\mathcal{P}_{0}^1 \cap \left[0,t-\left\lfloor\frac{t}{K}\right\rfloor K\right] 
= \emptyset\right)\right)^{|R|}\right)^{|\mathcal{X}^{y(\gamma),k(\gamma)}|}.
\]
Furthermore, since $t-\left\lfloor\frac{t}{K}\right\rfloor K \leq K$ and 
$q \geq q_K = 1+\frac{1}{3K|R|}\ln(1-e^{-K})$, 
\[
\mathds{P}_{q',q}\left( \mathcal{P}_{0}^1 \cap \left[0,t-\left\lfloor\frac{t}{K}\right\rfloor K\right] 
= \emptyset\right) = e^{-(1-q)\left(t-\left\lfloor\frac{t}{K}\right\rfloor K\right)}
\geq e^{\frac{1}{3K|R|}\ln(1-e^{-K})K}
  = (1-e^{-K})^{\frac{1}{3|R|}} 
\geq \left(\frac{1}{2}\right)^{\frac{1}{3|R|}}
\]
since $K \geq 2$. This implies 
\[
  \mathds{P}_{q',q}(\forall r \in \mathcal{X}^{y(\gamma),k(\gamma)}, 
  W^{\gamma,\hat{\eta}}(r)^c | \mathcal{F}) 
\leq \left(1-\mathds{P}_{q',q}\left(\hat{\eta}_0(0) = 0\right)^{|R|}
\left(\frac{1}{2}\right)^{\frac{1}{3}}\right)^{|\mathcal{X}^{y(\gamma),k(\gamma)}|}.
\]
In addition, if $\hat{\eta}=\eta$, $\mathds{P}_{q',q}(\hat{\eta}_0(0) = 0) = q'$, so 
$1-\mathds{P}_{q',q}(\eta_0(0) = 0)^{|R|} (\frac{1}{2})^{\frac{1}{3}} = 1-(q')^{|R|}2^{-\frac{1}{3}}$, 
and if $\hat{\eta}=\tilde{\eta}$, $1-\mathds{P}_{q',q}(\hat{\eta}_0(0) = 0)^{|R|} (\frac{1}{2})^{\frac{1}{3}} 
= 1-q^{|R|}2^{-\frac{1}{3}}$. Moreover, since 
$K \geq 2$, $q \geq q_K = 1+\frac{1}{3K|R|}\ln(1-e^{-K}) \geq 1+ \frac{1}{6|R|}\ln(1-e^{-2}) \geq \frac{1}{2}$, 
hence $1-\mathds{P}_{q',q}(\tilde{\eta}_0(0) = 0)^{|R|} (\frac{1}{2})^{\frac{1}{3}} 
\leq 1-2^{-|R|-\frac{1}{3}}$. This implies that if $\breve{c}_3'$ is the minimum of 
$-\ln(1-(q')^{|R|}2^{-\frac{1}{3}})$ and $-\ln(1-2^{-|R|-\frac{1}{3}})$ 
(which depends only on $\mathcal{U}$ and $q'$), for $\hat{\eta}=\eta$ or $\tilde{\eta}$ we have 
$\mathds{P}_{q',q}(\forall r \in \mathcal{X}^{y(\gamma),k(\gamma)}, 
  W^{\gamma,\hat{\eta}}(r)^c | \mathcal{F}) 
\leq e^{-\breve{c}_3'|\mathcal{X}^{y(\gamma),k(\gamma)}|}$. Consequently, (\ref{eq_wonderful_rectangles}) yields
\[
\mathds{P}_{q',q}\left(\left\{|\mathcal{X}^{y(\gamma),k(\gamma)}| > \frac{t}{6K}\right\} \cap
  \{\forall r \in \mathcal{X}^{y(\gamma),k(\gamma)}, 
  W^{\gamma,\hat{\eta}}(r)^c\}\right) 
  \]
  \[
  \leq \mathds{E}_{q',q}\left(\mathds{1}_{\{|\mathcal{X}^{y(\gamma),k(\gamma)}| > \frac{t}{6K}\}}
  e^{-\breve{c}_3'|\mathcal{X}^{y(\gamma),k(\gamma)}|}\right) \leq e^{-\breve{c}_3'\frac{t}{6K}}, 
\]
which is the lemma.
\end{proof}

\section*{Acknowledgements}

I would like to thank my PhD advisor Cristina Toninelli; 
I also would like to thank Ivailo Hartarsky for his careful reading of this paper and for his suggestions, as 
well as for pointing me to some references.

\end{document}